\documentclass[preprint,12pt]{elsarticle}

\usepackage{amssymb}
\usepackage{amsthm}

\newtheorem{theorem}{Theorem}
\newtheorem{lemma}{Lemma}

\newtheorem{corollary}{Corollary}

\newtheorem{remark}{Remark}

\usepackage{amsmath}
\usepackage{caption}
\usepackage{cases}
\usepackage[hidelinks]{hyperref}
\hypersetup{
    colorlinks = true,
    linkcolor = blue,
    citecolor = red,
    urlcolor = cyan,
}
\usepackage{float}
\usepackage{siunitx}
\usepackage{subcaption}

\numberwithin{equation}{section}

\journal{Applied Numerical Mathematics}

\begin{document}

\begin{frontmatter}

%% Title, authors and addresses

%% use the tnoteref command within \title for footnotes;
%% use the tnotetext command for theassociated footnote;
%% use the fnref command within \author or \address for footnotes;
%% use the fntext command for theassociated footnote;
%% use the corref command within \author for corresponding author footnotes;
%% use the cortext command for theassociated footnote;
%% use the ead command for the email address,
%% and the form \ead[url] for the home page:
%% \title{Title\tnoteref{label1}}
%% \tnotetext[label1]{}
%% \author{Name\corref{cor1}\fnref{label2}}
%% \ead{email address}
%% \ead[url]{home page}
%% \fntext[label2]{}
%% \cortext[cor1]{}
%% \affiliation{organization={},
%%             addressline={},
%%             city={},
%%             postcode={},
%%             state={},
%%             country={}}
%% \fntext[label3]{}

\title{Legendre Expansions of Products of Functions with Applications to Nonlinear Partial Differential Equations}

%% use optional labels to link authors explicitly to addresses:
%% \author[label1,label2]{}
%% \affiliation[label1]{organization={},
%%             addressline={},
%%             city={},
%%             postcode={},
%%             state={},
%%             country={}}
%%
%% \affiliation[label2]{organization={},
%%             addressline={},
%%             city={},
%%             postcode={},
%%             state={},
%%             country={}}

\author[inst1]{Rabia Djellouli}

\affiliation[inst1]{organization={Department of Mathematics and Interdisciplinary Research Institute for the Sciences},%Department and Organization
            addressline={California State University}, 
            city={Northridge},
            postcode={91330}, 
            state={CA},
            country={US}}

\author[inst1]{David Klein}
\author[inst2]{Matthew Levy}

\affiliation[inst2]{organization={Department of Mathematics},%Department and Organization
            addressline={California State University}, 
            city={Northridge},
            postcode={91330}, 
            state={CA},
            country={US}}

\begin{abstract}
%% Text of abstract
Given the Fourier--Legendre expansions of $f$ and $g$, and mild conditions on $f$ and $g$, we derive the Fourier--Legendre expansion of their product in terms of their corresponding Fourier--Legendre coefficients. We establish upper bounds on rates of convergence. We then employ these expansions to solve semi-analytically a class of nonlinear PDEs with a polynomial nonlinearity of degree 2. The obtained numerical results illustrate the efficiency and performance accuracy of this Fourier--Legendre-based solution methodology for solving a class of nonlinear PDEs.
\end{abstract}

%%Graphical abstract
% \begin{graphicalabstract}
% \includegraphics{grabs}
% \end{graphicalabstract}

%%Research highlights
% \begin{highlights}
% \item Research highlight 1
% \item Research highlight 2
% \end{highlights}

\begin{keyword}
%% keywords here, in the form: keyword \sep keyword
Fourier--Legendre series \sep Legendre polynomial \sep diffusion equation \sep nonlinear PDEs \sep nonlinear ODEs \sep Runge-Kutta methods
%% PACS codes here, in the form: \PACS code \sep code
%\PACS 0000 \sep 1111
%% MSC codes here, in the form: \MSC code \sep code
%% or \MSC[2008] code \sep code (2000 is the default)
\MSC[2020] 42C10 \sep 41A25  \sep 65L06 \sep 65N35  \sep 40-08 
\end{keyword}

\end{frontmatter}

\section{Introduction}
\label{sec:intro}

Fourier-type series have proven to be powerful tools for solving (semi-) analytically a large  class of linear PDEs that arise in a range of applications, including wave propagation, heat diffusion, compressible and incompressible fluid mechanics, free surface flow, flow in porous media, particle flow, strength of materials, elasticity, structural dynamics, transport engineering, and electrical engineering.  Fourier-type series have also been incorporated into a variety of numerical strategies for solving linear PDE formulation problems. For example, they have been used to construct absorbing boundary conditions (exact or approximate) to reformulate exterior electromagnetic or acoustic scattering problems in bounded domains (see, for example, \citep{Turkel} and references therein). 
In addition, they have been extensively employed to generate exact solutions to serve as reference solutions for assessing the performance accuracy of numerical methods designed to solve PDEs (see, for example, \citep{Barucq1}).

Solution methodologies that are based on expanding the solutions of the considered linear PDEs in terms of Fourier-type series including trigonometric, Legendre, Chebyshev, or Bessel series are often involved with the application of spectral methods. This class of methods was first introduced in the early 70s (the first paper was authored by Orszag in 1969 \citep{Orszag}) and emerged since the 90s as an alternative and effective computational tool to the ubiquitous finite element (FEM) and finite difference (FED) methods in scientific computing, as attested by the hundreds of publications in various applications (see, e.g., \citep{Gottlieb, Iliev, Boyd, Peyret, Canuto, Hesthaven, Shen, Hiptmair, Amara2012, Amara2014, Yu2021} and the references therein). These methods have been extensively investigated from both mathematical and numerical viewpoints. Results pertaining to convergence, stability, error estimates, boundary conditions and uniqueness, etc. can be found in the monographs \citep{Gottlieb, Iliev, Boyd, Peyret, Canuto, Hesthaven, Shen}, among others. The orthogonal property of the basis functions is the key feature of spectral methods that significantly reduces their implementation complexity and enhance their computational cost-effectiveness, but more remarkably it makes them outperform in terms of accuracy FEM and FED, particularly for smooth solutions in domains with simple geometries in which spectral methods have exponential convergence.

By contrast, for nonlinear problems, Fourier-type series have so far {\it only} been included in already established solution methodologies. Indeed, as is well-known \citep{DautrayLions}, standard methods for solving nonlinear problems involve two-step approaches, where in the first step, the nonlinear problem under consideration is formulated as a set or sequence of linear problems via linearization processes such as gradient methods, fixed-point, or path following approaches. In the second step, the linear problems from the first step are solved using approximation methods employing various representations such as polynomial bases, spectral expansions, wavelets, or reduced bases \citep{CiarletLions}. There is a second class of methods that  uses approximation prior to the linearization process. The main difference between the two approaches is that this latter class of methods gives rise to a sequence of linear problems in finite dimensional vector spaces whereas the standard approaches involve infinite dimensional spaces (such as Sobolev spaces) prior to the application of the discretization schemes. We are not aware of any work in which Fourier-type expansions are employed as the primary tool for solving the nonlinear PDEs. The main obstacle is that the presence of the nonlinearity prevents a significant gain in terms of computational complexity.
 
In this paper, our focus is on Fourier--Legendre series \citep{Stone, AbramowitzStegun}.  Fourier--Legendre series have found wide applications across scientific disciplines, for example, in the theory of the hypergeometric functions and complete elliptic integrals \citep{Campbell}, in the theory of the hypergeometric functions and fractional operators \cite{series}, for derivations and proofs of convergent series to $1/\pi$ and $1/\pi^2$ \citep{pi}, and in chemical physics \citep {Ragot}. Here, we prove new theorems on products of Fourier--Legendre series and show how they can be exploited to approximate solutions to a class of nonlinear PDEs in which the nonlinearity is a polynomial of degree 2.  The method is generalizable to polynomials of arbitrary degree, and it is particularly well suited to PDEs with diffusion terms, since the diffusion operator is diagonal in the orthogonal basis of Legendre polynomials in $L^2[-1,1]$.

More specifically, we show how the Fourier--Legendre coefficients of a product of two functions, $f$ and $g$, can be expressed in terms of the respective Fourier--Legendre coefficients of $f$ and $g$, and we derive an easily computable partial finite sum.  The key ingredient is a well-known combinatorial formula that expresses the product of two Legendre polynomials as an explicit linear combination of Legendre polynomials \citep{Adams, Bailey, Dougall, Salam}. This is carried out in Section \ref{sec:product}.  Section \ref{sec:conv_rates} establishes upper bounds of rates of convergence for product series approximations, depending on the smoothness of the factors $f$ and of $g$. In Section \ref{sec:application}, we illustrate how our results may be applied to find semi-analytic solutions to a class of nonlinear partial differential equations with diffusion, and quadratic polynomial nonlinearity.  In Section \ref{sec:discussion}, we summarize and assess our results, and identify further applications. Two appendices are also included. \ref{sec:appendix} includes data related to Section \ref{sec:discussion} and \ref{sec:appendixB} analyzes a function that fails to satisfy the hypotheses of Theorem \ref{L2}, but satisfies its conclusion, demonstrating that the hypotheses are sufficient but not necessary.

\section{Product Theorem}
\label{sec:product}

The Legendre polynomials, $\{P_0(x), P_1(x), P_2(x), \dots\}$, constitute a complete orthogonal basis of $ L^2 ([-1,1])$ and may be defined by the Rodrigues formula \citep{AbramowitzStegun},
\begin{equation}
P_n(x)=\frac{1}{2^n n!} \frac{d^n}{dx^n}(x^2-1)^n.
\end{equation}
Thus, $P_0(x)=1, P_1(x)=x, P_2(x)= (3x^2-1)/2$, etc., and they satisfy the orthogonality relationship,
\begin{equation}
\int_{-1}^{1} P_n(x)P_m(x)dx= \frac{2}{2n+1} \delta_{nm},
\end{equation}
 where $\delta_{nm}$ is the Kronecker delta function.

Assume that $f, g, f\cdot g\in L^2 ([-1,1])$ and denote the expansions of $f$ and $g$ in the orthogonal basis $\{P_n(x)\}$ by,
\begin{equation}    \label{expansion}
f(x)=\sum_{n=0}^{\infty}\alpha_n P_n(x)\qquad g(x)=\sum_{m=0}^{\infty}\beta_m P_m(x)
\end{equation}
where, %\citep{Canuto} (Eq. 2.3.9, pg 76),
\begin{equation}
\alpha_n=\frac{2n+1}{2}\int_{-1}^{1}f(x)P_n(x) dx, \qquad \beta_m=\frac{2m+1}{2}\int_{-1}^{1}g(x)P_m(x) dx.
\end{equation}
Since the function $f\cdot g$ is in $L^2([-1,1])$, it also has a Fourier--Legendre expansion,

\begin{equation}    \label{fg}
(f\cdot g)(x) = \sum_{k=0}^{\infty} \mu_k P_k(x).
\end{equation} 
Assuming pointwise and absolute convergence of Eqs \eqref{expansion}, we show in this section how to express each coefficient $\mu_k$ as a function of the coefficients $\{\alpha_n\}_{n\in\{0,1,2,\dots\}}$ and $\left\{ \beta_m \right\}_{m\in\{0,1,2,\dots\}}$ in the form of an infinite series, by first finding finite sum approximations.

Unless otherwise indicated, we assume throughout that the series in Eqs.\eqref{expansion} converge in the sense of $L^2 ([-1,1])$ and with pointwise absolute convergence. This is the case, for example, when $f$ and $g$ are absolutely continuous on $[-1,1]$ and $f^\prime$ and $g^\prime$ are of bounded variation, conditions that hold for the case that $f$ and $g$ are  continuously differentiable \citep{Saxena}.  Under these assumptions, we can express the product of the two series in Eq.\eqref{expansion} as a sum of products of Legendre polynomials as follows, 

\begin{equation}\label{product1}
(f\cdot g)(x)
=
\left(
\sum_{n=0}^{\infty} \alpha_n P_n(x)
\right)
\left(
\sum_{m=0}^{\infty}\beta_n P_m(x)
\right)
= \sum_{n=0}^{\infty}
\sum_{\ell=0}^{n}\alpha_{n-\ell}\beta_{\ell}P_{n-\ell}(x)P_{\ell}(x).
\end{equation}
Let $N$ be a positive integer. We define the following two convenient sequences: 
\begin{equation}\label{alphabeta}
\alpha_n^N=
\begin{cases}
\alpha_n &\text{if}\,\, n\leqslant N\\
 0&\text{otherwise} 
\end{cases}\qquad\qquad
\beta_n^N=
\begin{cases}
\beta_n &\text{if}\,\, n\leqslant N\\
 0&\text{otherwise}. 
\end{cases}
\end{equation}
Substituting $\alpha_n^N$ and $\beta_n^N$ for $\alpha_n$ and $\beta_n$ respectively in Eq.\eqref{product1} gives,
\begin{align}\label{product2}
\left(
\sum_{n=0}^{N} \alpha_n P_n(x)
\right)
\left(
\sum_{m=0}^{N}\beta_n P_m(x)
\right)
&=\left(
\sum_{n=0}^{\infty} \alpha_n^N P_n(x)
\right)
\left(
\sum_{m=0}^{\infty}\beta_n ^N P_m(x)
\right)\\
&= \sum_{n=0}^{\infty}
\sum_{\ell=0}^{n}\alpha_{n-\ell}^N\beta_{\ell}^N P_{n-\ell}(x)P_{\ell}(x)\label{zeroterms}\\
&= \sum_{n=0}^{2N}
\sum_{\ell=0}^{n}\alpha_{n-\ell}^N\beta_{\ell}^NP_{n-\ell}(x)P_{\ell}(x),
\end{align}
where the last step follows because all terms in Eq.\eqref{zeroterms} with $n>2N$ are zero.
 
In addition, assuming that both series in Eqs.\eqref{expansion} converge uniformly and that $f$ or $g$ is bounded on $[-1,1]$ we have,
\begin{align}
\label{product3}
(f\cdot g)(x)
&=\lim_{N\to\infty}\left[\left(
\sum_{n=0}^{N} \alpha_n P_n(x)
\right)
\left(
\sum_{m=0}^{N}\beta_n P_m(x)
\right)\right]\\
&=\lim_{N\to\infty} \sum_{n=0}^{2N}
\sum_{\ell=0}^{n}\alpha_{n-\ell}^N\beta_{\ell}^NP_{n-\ell}(x)P_{\ell}(x)
\end{align}
and the convergence is uniform.

To proceed further, we use the fact that the product of any two Legendre polynomials can be expressed as a linear combination of Legendre polynomials according to this formula (see for example \citep{Salam}),
\begin{equation}
\begin{split}
P_m(x)P_n(x)&=\\
\sum_{j=0}^{\min(m,n)}&\frac
{\left(\frac{1}{2}\right)_j\left(\frac{1}{2}\right)_{m-j}\left(\frac{1}{2}\right)_{n-j}(m+n-j)!}
{j!(m-j)!(n-j)!\left(\frac{3}{2}\right)_{m+n-j}}
\left(2(m+n-2j)+1\right)P_{m+n-2j}(x),
\end{split}
\end{equation}
where $(a)_r$ is the rising factorial function given by,
\begin{equation*}
(a)_0=1, \qquad (a)_r=a(a+1)(a+2)\dots(a+(r-1)), \qquad r=1,2,3,\dots
\end{equation*}
Thus,
\begin{align}   \label{begin-products}
&\sum_{\ell=0}^{k}\alpha_{k-\ell}\beta_{\ell}P_{k-\ell}(x)P_{\ell}(x) =\sum_{\ell=0}^{k}\alpha_{k-\ell}\beta_{\ell}\,\,\times\nonumber\\
&\left[
\sum_{j=0}^{\min(k-\ell,\ell)}
\frac
{
	\left(\frac{1}{2}\right)_j
	\left(\frac{1}{2}\right)_{k-\ell-j}
	\left(\frac{1}{2}\right)_{\ell-j}
	(k-j)!
}
{
	j!(k-\ell-j)!
	(\ell-j)!
	\left(\frac{3}{2}\right)_{k-j}
}
\left(2(k-2j)+1)\right)
P_{k-2j}(x)
\right].
\end{align}
For simplicity of notation, define $A_{jk\ell}$ by,
\begin{equation}    \label{ajkl}
A_{jk\ell} =
\frac{
	\left( \frac{1}{2} \right)_j
	\left( \frac{1}{2} \right)_{k-\ell-j} 
	\left( \frac{1}{2} \right)_{\ell-j} (k-j)!
}{
	j! (k-\ell-j)!(\ell-j)!
	\left( \frac{3}{2} \right)_{k-j}
}
\left( 2(k-2j)+1 \right),
\end{equation}
so that,
\begin{equation}\label{productsum}
\sum_{\ell=0}^{k}\alpha_{k-\ell}\beta_{\ell}P_{k-\ell}(x)P_{\ell}(x)=\sum_{\ell=0}^{k} \sum_{j=0}^{\min(k-\ell,\ell)}
	\alpha_{k-\ell}\beta_{\ell}A_{jk\ell} P_{k-2j}(x),
\end{equation}
with an analogous equation holding when $\alpha_n^N$ and $\beta_n^N$ are substituted for $\alpha_n$ and $\beta_n$ respectively.
 
\begin{remark} \label{Abound2}
$A_{jk\ell}$ is undefined unless $j\leq \min(k-\ell,\ell)$ because of the terms $(k-\ell-j)!$ and $(\ell-j)!$ in the denominator of Eq.\eqref{ajkl}. We note, however, that since the reciprocal of the Gamma function, $1/\Gamma (z)$, is entire, $A_{jk\ell}$ may be understood to take the value $0$ when these inequalities are violated.
\end{remark}

\begin{lemma} \label{Abound}
 For any integers $j,k,\ell$ satisfying  $0\leq j\leq\min(k-\ell,\ell)\leq\ell\leq k$, 
 $$0<A_{jk\ell}\leq 1,$$ 
 where $A_{jk\ell}$ is given by Eq.\eqref{ajkl}.
\end{lemma}
\begin{proof}
Set $x=1$ so that  $P_n(x)=1$ for all $n=0,1,2\dots$, in Eq.\eqref{productsum}. Now choose any $k_0$ and any $\ell_0 \leq k_0$.  Choose functions $f$ and $g$ so that  
$\beta_{\ell} = 1$ if $\ell = \ell_0$ and $\beta_{\ell} = 0$ otherwise, and so that $\alpha_n = 1$ for $n=0, 1, \dots, k_0$. Then Eq.\eqref{productsum} becomes,
\begin{equation}\label{productsum'}
1=\sum_{j=0}^{\min(k_0-\ell_0,\ell_0)}A_{jk_0\ell_0}. 
\end{equation}
Since each summand in Eq.\eqref{productsum'} is positive, it follows that each is bounded by $1$. 
\end{proof}
 Next, we combine Eqs\eqref{product1}, \eqref{begin-products} and \eqref{productsum} to deduce that,
	\begin{equation}    \label{square-finite}
	(f\cdot g)(x)
	=
	\sum_{k=0}^{\infty} \sum_{\ell=0}^{k} \sum_{j=0}^{\min(k-\ell,\ell)}
	\alpha_{k-\ell}\beta_{\ell}A_{jk\ell} P_{k-2j}(x),
	\end{equation}
and similarly using Eq.\eqref{product3}, we obtain,
\begin{equation}    \label{square-finite2}
	(f\cdot g)(x)
	=
	\lim_{N\to\infty}\sum_{k=0}^{2N} \sum_{\ell=0}^{k} \sum_{j=0}^{\min(k-\ell,\ell)}
	\alpha_{k-\ell}^N\beta_{\ell}^N A_{jk\ell} P_{k-2j}(x).
	\end{equation}
Again, under the assumption that the series in Eq.\eqref{expansion} converge uniformly and that $f$ or $g$ is bounded on $[-1,1]$, the convergence in Eq.\eqref{square-finite} is uniform.

The next lemma enables us to make a change of dummy variables and then interchange the inner two sums in Eqs. \eqref{square-finite} and \eqref{square-finite2}.  
\begin{lemma}\label{chi}
	Let $\chi_A (j)$ denote the indicator function on nonnegative integers which takes the value $1$ if $j\in A$ and $0$ otherwise. For nonnegative integers $j,k,\ell$,
	\[ \chi_{[0\leq\ell\leq k]}(\ell)\cdot\chi_{[0\leq j\leq \min(k-\ell,\ell)]}(j)=
		\chi_{[j\leq\ell\leq k-j]}(\ell)\cdot
	\chi_{[0\leq j\leq \frac{k}{2}]}(j)
	\]
	where the subscript $[0\leq\ell\leq k]$ stands for $\{\ell: 0\leq\ell\leq k\}$ and similarly for the other subscripts.
\end{lemma}
\begin{proof}
	Observe first that $\chi_{[0\leq j\leq \min(k-\ell,\ell)]}(j)= \chi_{[0\leq j\leq\ell]}(j)\cdot
	\chi_{[0\leq j\leq k-\ell]}(j)$.  Thus, it suffices to show that 
\begin{equation}\label{chis}
\chi_{[0\leq\ell\leq k]}(\ell)\cdot
	\chi_{[0\leq j\leq k-\ell]}(j)\cdot
	\chi_{[0\leq j\leq\ell]}(j)
	=
	\chi_{[j\leq\ell\leq k-j]}(\ell)\cdot
	\chi_{[0\leq j\leq \frac{k}{2}]}(j).
\end{equation}
We show that left side of Eq.\eqref{chis} $= 1$ if and only if the right side $= 1$.
	
Assume that the left side equals 1. Then $0\leq j\leq k- \ell$ and $j\leq\ell$. Adding these inequalities, it follows that $0\leq 2j\leq k$ and therefore $0\leq j\leq \frac{k}{2}$.  If the left side equals 1, it also follows that $ j\leq\ell$ and  $0\leq j\leq k-\ell$.  Therefore  $0\leq\ell\leq k-j$.  Thus, the right side of the equation $= 1$.

Next, assume that the right side of Eq.\eqref{chis} equals 1.  Then, $0\leq j\leq \frac{k}{2}$ and $j\leq\ell\leq k-j$.  It follows that $0\leq j\leq\ell$ and $0\leq j\leq k-\ell$ and thus $\chi_{[0\leq\ell\leq k]}\cdot\chi_{[0\leq j\leq k-\ell]}=1$.  Also, since $\ell\leq k-j$, then $\ell\leq k$ and $\chi_{[0\leq j\leq\ell]}=1$ and the right side of Eq.\eqref{chis} equals 1.
	\end{proof}
	
\begin{corollary} \label{cor1} 
Let $f$ and $g$ be given as in Eqs. \eqref{expansion} with pointwise absolute convergence.  Then, for each $x\in[-1,1]$, we have,
	\begin{equation}	\label{reorder-sums}
		(f\cdot g)(x)
		=
		\lim_{N\to\infty}\sum_{k=0}^{N} \sum_{j=0}^{\lfloor\frac{k}{2}\rfloor} \sum_{\ell=j}^{k-j}
		\alpha_{k-\ell} \beta_\ell A_{jk\ell} P_{k-2j}(x),
	\end{equation}
	where $\lfloor\frac{k}{2}\rfloor$ is the greatest integer less than or equal to $k/2$. If both series in Eq.\eqref{expansion} converge uniformly and $f$ or $g$ is bounded on $[-1,1]$, then
	\begin{equation}	\label{reorder-sums2}
		(f\cdot g)(x)
		=
		\lim_{N\to\infty}\sum_{k=0}^{2N} \sum_{j=0}^{\lfloor\frac{k}{2}\rfloor} \sum_{\ell=j}^{k-j}
		\alpha_{k-\ell}^N \beta_\ell^N A_{jk\ell} P_{k-2j}(x),
	\end{equation}
	with uniform convergence.
\end{corollary}
\begin{proof}
	From Lemma \ref{chi}, we have for any $k\geq 0$,
	\begin{align}
	\sum_{\ell=0}^{k} \sum_{j=0}^{\min(k-\ell,\ell)}
	\alpha_{k-\ell}&\beta_{\ell}A_{jk\ell} P_{k-2j}(x)\nonumber\\
		 =\sum_{\ell=0}^{\infty} \sum_{j=0}^{\infty}
		\chi_{[0\leq\ell\leq k]}&\cdot\chi_{[0\leq j\leq \min(k-\ell,\ell)]} \alpha_{k-\ell}\beta_{\ell} A_{jk\ell}P_{k-2j}(x)\nonumber
		\\
		&=
		\sum_{j=0}^{\infty} \sum_{\ell =0}^{\infty}
		\chi_{[j \leq\ell\leq k-j]}\chi_{[0\leq j\leq\lfloor\frac{k}{2}\rfloor]} \alpha_{k-\ell}\beta_{\ell} A_{jk\ell}P_{k-2j}(x)\nonumber
		\\
		&=
		\sum_{j=0}^{\lfloor\frac{k}{2}\rfloor} \sum_{\ell=j}^{k-j}
		\alpha_{k-\ell} \beta_\ell A_{jk\ell} P_{k-2j}(x).
	\end{align}
Combining with Eq.\eqref{square-finite} gives,
	\begin{align*}
		(f\cdot g)(x)
		=
		\sum_{k=0}^{\infty} \sum_{j=0}^{\lfloor\frac{k}{2}\rfloor} \sum_{\ell=j}^{k-j}
		\alpha_{k-\ell} \beta_\ell A_{jk\ell} P_{k-2j}(x)
	\end{align*}
with pointwise convergence.  Uniform convergence of Eq.\eqref{reorder-sums2} follows from the uniform convergence of Eq.\eqref{square-finite2}.
\end{proof}

\begin{theorem} \label{main}
Let $f$ and $g$ be given as in Eq. \eqref{expansion}.  Then,

(a) For each $N=0,1,2\dots$,
\begin{align}\label{sum=sum}
\sum_{k=0}^{N} \sum_{\ell=0}^{k} \sum_{j=0}^{\min(k-\ell,\ell)}
	\alpha_{k-\ell}&\beta_{\ell}A_{jk\ell} P_{k-2j}(x)
		\\\nonumber
		= 
		&\sum_{n=0}^{N}
		\left[
			\sum_{m=0}^{\lfloor\frac{N-n}{2}\rfloor} \sum_{\ell=m}^{n+m}
			\alpha_{n+2m-\ell}\beta_\ell A_{m,n+2m,\ell}
		\right]
		P_n(x).
		\end{align}

(b) Assuming pointwise absolute convergence in Eq. \eqref{expansion}, 
\begin{equation}	\label{change-var}
		(f\cdot g)(x)
		=
		\lim_{N\to\infty} \sum_{n=0}^{N}
		\left[
			\sum_{m=0}^{\lfloor\frac{N-n}{2}\rfloor} \sum_{\ell=m}^{n+m}
			\alpha_{n+2m-\ell}\beta_\ell A_{m,n+2m,\ell}
		\right]
		P_n(x),
	\end{equation}
	for each $x\in[-1,1]$.
	
(c) If the series in Eq.\eqref{expansion} converge uniformly and $f$ or $g$ is bounded on $[-1,1]$, then,
\begin{equation}	\label{change-var2}
		(f\cdot g)(x)
		=
		\lim_{N\to\infty} \sum_{n=0}^{2N}
		\left[
			\sum_{m=0}^{\lfloor\frac{2N-n}{2}\rfloor} \sum_{\ell=m}^{n+m}
			\alpha_{n+2m-\ell}^N\beta_\ell^N A_{m,n+2m,\ell}
		\right]
		P_n(x),
	\end{equation}
	and the convergence is uniform.
	\end{theorem}

\begin{proof}  For Part (a), observe from the proof of Corollary \ref{cor1} that for each $N$,
\begin{equation}\label{parta1}
\sum_{k=0}^{N} \sum_{\ell=0}^{k} \sum_{j=0}^{\min(k-\ell,\ell)}
	\alpha_{k-\ell}\beta_{\ell}A_{jk\ell} P_{k-2j}(x) = 
		\sum_{k=0}^{N} \sum_{j=0}^{\lfloor\frac{k}{2}\rfloor} \sum_{\ell=j}^{k-j}
		\alpha_{k-\ell} \beta_\ell A_{jk\ell} P_{k-2j}(x).		
		\end{equation}
In order to make a change of variables for the first two sums, $\sum\limits_{k=0}^{N}\sum\limits_{j=0}^{\lfloor\frac{k}{2}\rfloor}$, on the right side of Eq. \eqref{parta1}, we define a linear transformation  $T: \mathbb{R}^2 \to \mathbb{R}^2$ by
\begin{equation}	\label{transform}
(n,m)=T(k,j)=(k-2j,j).
\end{equation}
Note that, in this expression, the variables $j, k, n, m$ are real numbers, but below we will restrict them to take integer values. Since $T$ is a linear homeomorphism, it maps line segments to line segments and boundaries and interiors respectively of closed sets to boundaries and interiors respectively of closed sets, in the plane.

The the first two sums $\sum\limits_{k=0}^{N}\sum\limits_{j=0}^{\lfloor\frac{k}{2}\rfloor}$ on the right side of Eq. \eqref{parta1} can be expressed as a finite sum of the points $(k,j)$ in the triangular region in $\mathbb{R}^2\cap\mathbb{Z}^2$
whose boundary in $\mathbb{R}^2$ is the triangle consisting of the union of the following three line segments:
\begin{equation*}
	\{(k,0) \mid 0\leq k\leq N\},
	\,\,
	\left \{
		(k,j) \mathrel{}\middle|\mathrel{}
		0\leq k\leq N, j=\frac{k}{2}
	\right \},
	\,\,
	\left \{
		(N,j) \mathrel{}\middle|\mathrel{}
		0\leq j\leq \frac{N}{2}
	\right \}.
\end{equation*}
Under the transformation $T$ of (\ref{transform}), the above triangular region is mapped to the region whose boundary is the triangle consisting of the union of the following line segments:
\begin{equation*}
	\{(n,0) \mid 0\leq n\leq N\},\quad
	\left \{
		(0,m) \mathrel{}\middle|\mathrel{}
		0\leq m\leq\frac{N}{2}
	\right \},
	\end{equation*}
	and
	\begin{equation*}
	\left \{
		(n,m) \mathrel{}\middle|\mathrel{}
		0\leq n\leq N, m=\frac{N-n}{2}
	\right \}.
\end{equation*}
Therefore, with the change of variables, $n=k-2j$ and $m=j$, we can make the substitutions,
	\begin{equation*}
		k=n+2j=n+2m,
		\qquad
		k-j=n+j=n+m,
	\end{equation*}
	and find that,
	\begin{align}
		\sum_{k=0}^{N} \sum_{j=0}^{\lfloor\frac{k}{2}\rfloor} \sum_{\ell=j}^{k-j}
		\alpha_{k-\ell}& \beta_\ell A_{jk\ell} P_{k-2j}(x)
		\\
		=&
		 \sum_{n=0}^{N}
		\left[
		\sum_{m=0}^{\lfloor\frac{N-n}{2}\rfloor} \sum_{\ell=m}^{n+m}
		\alpha_{n+2m-\ell} \beta_\ell A_{m,n+2m,\ell} 
		\right]
		P_{n}(x),\nonumber
	\end{align}
thus establishing Part (a). Parts (b) and (c) now follow from Part (a) and Corollary \ref{cor1}.
\end{proof}
We will make use of the following result from Wang \citep{Wang3}. 

\begin{theorem} \label{wangtheorem} (Wang \citep{Wang3}) Suppose that $f,f^{\prime},\dots, f^{(j-1)}$ are absolutely continuous on $[-1,1]$ and $f^{(j)}$ has bounded variation.  Then, using the notation of Eq.\eqref{expansion}, for $n\geq j+1$ and $j\geq1$,

\begin{equation}\label{alphabound2}
|\alpha_n|\leq\frac{\sqrt{2/\pi}||f^{(j)}||}{\sqrt{n-j}\,(n-\frac{1}{2})(n-\frac{3}{2})\cdots(n-\frac{2j-1}{2})},
\end{equation}
where $||f^{(j)}||$ is the total variation of $f^j$ on the interval $[-1,1]$. If $j=0$, then $|\alpha_n|\leq\sqrt{2/\pi}||f||/\sqrt{n}$ for $n\geq1$.
\end{theorem}

We note that in the case that $f^j$ is absolutely continuous, $||f^j||=$\\
$\int_{-1}^1|f^{(j+1)}(x)|dx$. 

\begin{corollary}\label{newbounds} Suppose that $f,f^{\prime},\dots, f^{(j-1)}$ are absolutely continuous on $[-1,1]$ and $f^{(j)}$ has bounded variation.  Then, using the notation of Eq.\eqref{expansion}, for $n\geq j+1$ and $j\geq1$,

\begin{equation}\label{alphabound3}
|\alpha_n|\leq\frac{\sqrt{2/\pi}||f^{(j)}||}{\sqrt{n-j}\,(n-\frac{1}{2})(n-\frac{3}{2})\cdots(n-\frac{2j-1}{2})}\leq\frac{\sqrt{2/\pi}||f^{(j)}||}{\left(n-j\right)^{(2j+1)/2}}.
\end{equation}
\end{corollary}
\begin{proof}
The lemma follows immediately from Theorem \ref{wangtheorem} and the observation that,
\begin{equation}
\sqrt{n-j}\left(n-\frac{1}{2}\right)\left(n-\frac{3}{2}\right)\cdots\left(n-\frac{2j-1}{2}\right)\geq (n-j)^{\frac{2j+1}{2}}.
\end{equation}
\end{proof}

\begin{lemma}\label{l2bound}
Suppose that $f$ and $g$ are absolutely continuous on $[-1,1]$ and $f^\prime$ and $g^\prime$ are of bounded variation.  
Then, the collection
\begin{equation*}
\left\{\sum_{n=0}^{N}
		\left[
			\sum_{m=0}^{\lfloor\frac{N-n}{2}\rfloor} \sum_{\ell=m}^{n+m}
			\alpha_{n+2m-\ell}\beta_\ell A_{m,n+2m,\ell}
		\right]
		P_n(x)\right\}
\end{equation*}
is uniformly bounded in $N$ and $x$.
\end{lemma}
\begin{proof}  From Eqs.\eqref{productsum} and \eqref{sum=sum}, we have,
\begin{align}\label{productsum2}
\sum_{n=0}^{N}\left[
			\sum_{m=0}^{\lfloor\frac{N-n}{2}\rfloor}\sum_{\ell=m}^{n+m}
			\alpha_{n+2m-\ell}\beta_\ell A_{m,n+2m,\ell}
		\right]&P_n(x)\nonumber\\
				=\sum_{n=0}^{N}\sum_{\ell=0}^{n} \sum_{j=0}^{\min(n-\ell,\ell)}
	\alpha_{n-\ell}\beta_{\ell}A_{jn\ell} &P_{n-2j}(x)\\
	=\sum_{n=0}^{N}\sum_{\ell=0}^{n}\alpha_{n-\ell}\beta_{\ell}P_{n-\ell}(x)&P_{\ell}(x).\nonumber
	\end{align}
Thus, it suffices to show that the collection,
\begin{equation*}
\left\{\sum_{n=0}^{N}
\sum_{\ell=0}^{n}|\alpha_{n-\ell}\beta_{\ell}P_{n-\ell}(x)P_{\ell}(x)|\right\}
\end{equation*}
is uniformly bounded in $N$ and in $x$.  Since $|P_n(x)| \leq 1$ for all $n$ and\\ $x\in [-1,1]$ \citep{Canuto} (Eq. 2.3.4), we show that the set of numbers,

\begin{equation*}
\left\{\sum_{n=0}^{N}
\sum_{\ell=0}^{n}|\alpha_{n-\ell}\beta_{\ell}|\right\}
\end{equation*}
is uniformly bounded in $N$.  Using Corollary \ref{newbounds} with $j=1$, we have the following bounds,
\begin{align}
|\alpha_n| &\leq\frac{A_1}{\left(n-1\right)^{3/2}}\label{alphabound}\\ 
|\beta_n| &\leq\frac{B_1}{\left(n-1\right)^{3/2}}\label{betabound},
\end{align}
for $n\geq2$, where $A_1$ and $B_1$ are constants depending on $f$ and $g$ respectively.  For $n\geq 6$, we may write
\begin{equation}\label{splitsum}
\sum_{\ell=0}^{n}|\alpha_{n-\ell}\beta_{\ell}| =\sum_{\ell=0}^{2}|\alpha_{n-\ell}\beta_{\ell}|+\sum_{\ell=0}^{2}|\beta_{n-\ell}\alpha_{\ell}|+\sum_{\ell=3}^{n-3}|\alpha_{n-\ell}\beta_{\ell}|
\end{equation}
and the sum on $n$ of the first two terms on the right side of Eq.\eqref{splitsum} is finite by Eqs.\eqref{alphabound} and \eqref{betabound}. Again, by Eqs.\eqref{alphabound} and \eqref{betabound}, the third term on the right side of Eq.\eqref{splitsum} is bounded as follows,
\begin{align}
\sum_{\ell=3}^{n-3}|\alpha_{n-\ell}\beta_{\ell}| &\leq  A_1B_1 \sum_{\ell=3}^{n-3}\frac{1}{\left(n-\ell-1\right)^{3/2}\left(\ell-1\right)^{3/2}}\label{sumn-3}\\
&< A_1B_1\int_{2}^{n-2}\frac{dx}{(n-x-1)^{3/2} (x-1)^{3/2}}\label{integraln-2}\\
&=A_1B_1\frac{4 (n-4)}{\sqrt{n-3} (n-2)^2},
\end{align}
which is summable on $n$. The integral in Eq.\eqref{integraln-2} is an upper bound for the sum in Eq.\eqref{sumn-3} because the integrand is a convex function with minimum at $n/2$.  Thus, the set
\begin{equation}
\left\{\sum_{n=0}^{N}
\sum_{\ell=0}^{n}|\alpha_{n-\ell}\beta_{\ell}|\right\}
\end{equation}
is uniformly bounded in $N$, and the lemma is proved.
\end{proof}

\begin{theorem}\label{L2}
Suppose that $f$ and $g$ are absolutely continuous on $[-1,1]$ and $f^\prime$ and $g^\prime$ are of bounded variation. Then, the Fourier--Legendre expansion of $f\cdot g$ in $L^2([-1,1])$ is given by,

 \begin{equation}	\label{change-var3}
		(f\cdot g)(x)
		=
		 \sum_{n=0}^{\infty}
		\left[
			\sum_{m=0}^{\infty} \sum_{\ell=m}^{n+m}
			\alpha_{n+2m-\ell}\beta_\ell A_{m,n+2m,\ell}
		\right]
		P_n(x).
\end{equation}
	
Thus, in the notation of Eq.\eqref{fg},
\begin{align}\label{muformula}
\mu_k =\sum_{m=0}^{\infty} \sum_{\ell=m}^{k+m}
			\alpha_{k+2m-\ell}\beta_\ell A_{m,k+2m,\ell}.
\end{align}
\end{theorem}
\begin{proof}
We first note that the inner product on $L^2 ([-1,1])$ given by,
\begin{equation}
\langle \phi, \psi\rangle = \int_{-1}^{1} \phi(x)\psi(x)\,dx
\end{equation}
is continuous in the $L^2 ([-1,1])$ Hilbert space topology. To simplify the notation, we set
 \begin{equation}	\label{phiN}
		\phi_N (x)= 
		 \sum_{n=0}^{N}
		\left[
		\sum_{m=0}^{\lfloor\frac{N-n}{2}\rfloor} \sum_{\ell=m}^{n+m}
		\alpha_{n+2m-\ell} \beta_\ell A_{m,n+2m,\ell} 
		\right]
		P_{n}(x).
	\end{equation}
By Theorem \ref{main}b, $(f\cdot g)(x) = \lim_{N\to\infty} \phi_N (x)$ pointwise, and by Lemma \ref{l2bound} the sequence$\{\phi_N(x)\}$ is uniformly bounded by a constant on $[-1,1]$. Thus, by the Dominated Convergence Theorem $(f\cdot g)(x) = \lim_{N\to\infty} \phi_N (x)$ in $L^2([-1,1])$.  Then, referring to Eq.\eqref{fg}, we find that,
\begin{align}
\frac{2}{2k+1}\mu_k =& \,\langle f\cdot g, P_k\rangle\\
= & \,\langle \lim_{N\to\infty} \phi_N, P_k\rangle\\
= & \,\lim_{N\to\infty}\langle  \phi_N, P_k\rangle\\
= & \,\frac{2}{2k+1}\lim_{N\to\infty} \sum_{m=0}^{\lfloor\frac{N-k}{2}\rfloor} \sum_{\ell=m}^{k+m}
		\alpha_{k+2m-\ell} \beta_\ell A_{m,k+2m,\ell},
\end{align}
where the last step follows from orthogonality of the Legendre Polynomials. Thus, 
\begin{align}\label{muformula2}
\mu_k =& \lim_{N\to\infty} \sum_{m=0}^{\lfloor\frac{N-k}{2}\rfloor} \sum_{\ell=m}^{k+m}
		\alpha_{k+2m-\ell} \beta_\ell A_{m,k+2m,\ell}\\
			=&\sum_{m=0}^{\infty} \sum_{\ell=m}^{k+m}
			\alpha_{k+2m-\ell}\beta_\ell A_{m,k+2m,\ell}~.
\end{align}
\end{proof}
\begin{remark}\label{hypotheses}
We give an example of a function in \ref{sec:appendixB} that fails to satisfy the hypotheses of Theorem \ref{L2}, but nevertheless satisfies its conclusion.
\end{remark}

\section{Rates of Convergence}
\label{sec:conv_rates}

In this section, we collect results on the rate of convergence of the Fourier--Legendre series for $f\cdot g$ and for the rates of convergence of the series in Eq. \eqref{muformula} for the coefficients $\{\mu_k \}$.  For the sake of notational simplicity, we adopt the following conventions,
\begin{align}
A_j=&\sqrt{2/\pi}||f^{(j)}||\\
B_j=&\sqrt{2/\pi}||g^{(j)}||\\
C_j=&\sqrt{2/\pi}||(fg)^{(j)}||
\end{align}
for $j=1,2,3\dots$ where, on the right hand side, we are using the notation of Theorem \ref{wangtheorem}.

 \begin{theorem} \label{newstone}
Suppose that $f$ and $g$ are absolutely continuous on $[-1,1]$ and $f^\prime$ and $g^\prime$ are of bounded variation.  Then, the following series converges uniformly: 
\begin{equation}	\label{boundthm1}
		(f\cdot g)(x)=\sum_{n=0}^{\infty}\mu_n P_n(x)
		=
		 \sum_{n=0}^{\infty}
		\left[
			\sum_{m=0}^{\infty} \sum_{\ell=m}^{n+m}
			\alpha_{n+2m-\ell}\beta_\ell A_{m,n+2m,\ell}
		\right]
		P_n(x).
	\end{equation}
	Moreover,
\begin{flalign}\label{boundthm2}
\Biggl|(f\cdot g)(x)- \sum_{n=0}^{N}&\left[\sum_{m=0}^{\infty} \sum_{\ell=m}^{n+m} \alpha_{n+2m-\ell}\beta_\ell A_{m,n+2m,\ell}\right]P_n(x)\Biggr| &\\
&\leq C_1\sum_{n=N+1}^{\infty}\frac{1}{\sqrt{n-1}\left(n-\frac{1}{2}\right)}
\end{flalign}
for any $N\geq 1$ and all $x\in[-1,1]$. In particular, for $N\geq 2$ and all $x\in[-1,1]$,
\begin{flalign}\label{boundthm3}
\Biggl|(f\cdot g)(x)- \sum_{n=0}^{N}&\left[\sum_{m=0}^{\infty} \sum_{\ell=m}^{n+m} \alpha_{n+2m-\ell}\beta_\ell A_{m,n+2m,\ell}\right]P_n(x)\Biggr|\nonumber &\\
&\leq C_1\sqrt{2} \left(\pi -2 \tan
   ^{-1}\left(\sqrt{2}
   \sqrt{N-1}\right)\right)\leq\frac{2C_1}{\sqrt{N-1}}.
\end{flalign}
\end{theorem}
\begin{proof}
Since the product of two absolutely continuous functions is absolutely continuous and products and sums of functions of bounded variation are of bounded variation \citep{royden}, it follows that $f\cdot g$ is absolutely continuous and $(f\cdot g)^\prime$ has bounded variation.  Thus, Theorem \ref{wangtheorem} applies to $f\cdot g$ with $j=1$.  Therefore,
\begin{flalign}
\Biggl|(f\cdot g)(x)- \sum_{n=0}^{N}&\left[\sum_{m=0}^{\infty} \sum_{\ell=m}^{n+m} \alpha_{n+2m-\ell}\beta_\ell A_{m,n+2m,\ell}\right]P_n(x)\Biggr|\nonumber&\\
&\leq \sum_{n=N+1}^{\infty}|\mu_n P_n(x)|\leq \sum_{n=N+1}^{\infty}|\mu_n|\\
&\leq C_1\sum_{n=N+1}^{\infty}\frac{1}{\sqrt{n-1}\left(n-\frac{1}{2}\right)}\nonumber,
\end{flalign}
where in the second line we have used the fact that $|P_n(x)|\leq 1$ for all $x$ and $n$.  Inequality \eqref{boundthm3} follows from direct calculation of the integrals in the following inequalities,
\begin{flalign}
\sum_{n=N+1}^{\infty}\frac{1}{\sqrt{n-1}\left(n-\frac{1}{2}\right)}\leq \int_{N}^{\infty}\frac{dn}{\sqrt{n-1}\left(n-\frac{1}{2}\right)}
\leq \int_{N}^{\infty}\frac{dn}{(n-1)^{\frac{3}{2}}},
\end{flalign}
\end{proof}

 As noted in the proof of Theorem \ref{newstone}, $f\cdot g$ is absolutely continuous and products and sums of functions of bounded variation are of bounded variation. It follows that $f\cdot g$ is absolutely continuous and $(f\cdot g)^\prime$ has bounded variation.  In light of this, the following result is a direct consequence of results in \citep{Wang3}.
 \begin{theorem} \label{newstone2}
Suppose that $f,f^{\prime},\dots, f^{(j-1)}$ and $g,g^{\prime},\dots, g^{(j-1)}$are absolutely continuous on $[-1,1]$ and $f^{(j)}$ and $g^{(j)}$  have bounded variation.  Then, with the notation of Eqs.\eqref{expansion}, \eqref{fg}, and Theorem \ref{wangtheorem}, for $N\geq j+1$ and $j\geq2$, 
\begin{flalign}\label{boundthm2_2}
\Biggl|(f\cdot g)(x)- \sum_{n=0}^{N-1}&\left[\sum_{m=0}^{\infty} \sum_{\ell=m}^{n+m} \alpha_{n+2m-\ell}\beta_\ell A_{m,n+2m,\ell}\right]P_n(x)\Biggr| &\\
&\leq \frac{C_j}{(j-1)\sqrt{N-j}}\prod_{k=2}^{j}\frac{1}{\left(N-\frac{2k-1}{2}\right)}
\end{flalign}
for all $x\in[-1,1]$. 
\end{theorem}

\begin{theorem}  \label{mubound} Suppose that $f,f^{\prime},\dots, f^{(j-1)}$ and $g,g^{\prime},\dots, g^{(j-1)}$are absolutely continuous on $[-1,1]$ and $f^{(j)}$ and $g^{(j)}$  have bounded variation.  Then, with the notation of Eqs.\eqref{expansion} and \eqref{fg}, for $M\geq j+1$ and $j\geq1$,

\begin{flalign}\label{theoremerror'}
\Biggl|\mu_k - \sum_{m=0}^{M}&\sum_{\ell=m}^{k+m}\alpha_{k+2m-\ell}\beta_\ell A_{m,k+2m,\ell}\Biggr|&\\
&\leq A_jB_j\sum_{m=M+1}^{\infty}\int_{m-1}^{k+m+1}\frac{\,dx}{(k+2m-x-j)^{\frac{2j+1}{2}}(x-j)^{\frac{2j+1}{2}}}.\nonumber		
\end{flalign}
\end{theorem}

\begin{proof}
From Eq.\eqref{alphabound3}, for $j+1\leq m\leq\ell\leq m+k$, we have,
\begin{equation}
|\alpha_{k+2m-\ell}|\leq\frac{A_j}{\left(k+2m-\ell-j\right)^{\frac{2j+1}{2}}},
\end{equation}
and,
\begin{equation}
|\beta_{\ell}|\leq\frac{B_j}{\left(\ell-j\right)^{\frac{2j+1}{2}}}.
\end{equation}
By Lemma \ref{Abound}, for $m\leq\ell\leq k+m$, $|A_{m,k+2m,\ell}|\leq 1$, so

\begin{flalign}
\sum_{\ell=m}^{k+m}
			|\alpha_{k+2m-\ell}\beta_\ell &A_{m,k+2m,\ell}| \leq \sum_{\ell=m}^{k+m} |\alpha_{k+2m-\ell}\beta_\ell|\nonumber\\		
			&\leq \sum_{\ell=m}^{k+m} \frac{A_jB_j}{\left(k+2m-\ell-j\right)^{\frac{2j+1}{2}}\left(\ell-j\right)^{\frac{2j+1}{2}}}\\
			& \leq \int_{m-1}^{k+m+1}\frac{A_jB_j\,dx}{(k+2m-x-j)^{\frac{2j+1}{2}} (x-j)^{\frac{2j+1}{2}}}\nonumber,
\end{flalign}
where in the last step, we have used the readily verified fact that the integrand is a convex function. Thus, from Eq.\eqref{muformula},

\begin{flalign}
\left|\mu_k - \sum_{m=0}^{M}\sum_{\ell=m}^{k+m}\alpha_{k+2m-\ell}\beta_\ell A_{m,k+2m,\ell}\right|
&\leq \sum_{m=M+1}^{\infty}\sum_{\ell=m}^{k+m} \left|\alpha_{k+2m-\ell}\beta_{\ell} A_{m,k+2m,\ell}\right|\\
\leq \sum_{m=M+1}^{\infty}\int_{m-1}^{k+m+1}&\frac{A_jB_j\,dx}{(k+2m-x-j)^{\frac{2j+1}{2}}(x-j)^{\frac{2j+1}{2}}}.\nonumber		
\end{flalign}
\end{proof}

\begin{corollary}\label{coeffcor} Suppose that $f$ and $g$ are absolutely continuous on $[-1,1]$ with derivatives of bounded variation.  Then, in the notation of  Eqs.\eqref{expansion}, \eqref{fg}, for $M\geq 3$,
\begin{flalign}\label{j=1,k=2}
&\frac{1}{A_1B_1}\left|\mu_k - \sum_{m=0}^{M}\sum_{\ell=m}^{k+m}\alpha_{k+2m-\ell}\beta_\ell A_{m,k+2m,\ell}\right|\nonumber&\\	
&\qquad\qquad\leq\frac{4 (k+2)}{(k+2 M-2) (2 (\sqrt{(M-2)
   (k+M)}+M-1)+k)}.&
\end{flalign} 
\end{corollary}
\begin{proof}
From Eq.\eqref{theoremerror'} with $j=1$,
\begin{flalign}\label{theoremerror2}
&\left|\mu_k - \sum_{m=0}^{M}\sum_{\ell=m}^{k+m}\alpha_{k+2m-\ell}\beta_\ell A_{m,k+2m,\ell}\right|\nonumber&\\
&\qquad\qquad\leq \sum_{m=M+1}^{\infty}\int_{m-1}^{k+m+1}\frac{A_1B_1\,dx}{(k+2m-x-1)^{\frac{3}{2}}(x-1)^{\frac{3}{2}}}& \\
&\qquad\qquad=4A_1B_1  \sum_{m=M+1}^{\infty}\frac{(k+2)}{\sqrt{(m-2) (k+m)} (k+2 m-2)^2}\nonumber&\\
&\qquad\qquad\leq 4A_1B_1 \int_{M}^{\infty}\frac{(k+2)dm}{\sqrt{(m-2) (k+m)} (k+2 m-2)^2},\nonumber&			
\end{flalign}
where in the last step, we require $M\geq 3$.  This last integral may be evaluated exactly as,
\begin{flalign}\label{theoremerror3}
& \int_{M}^{\infty}\frac{(k+2)dm}{\sqrt{(m-2) (k+m)} (k+2 m-2)^2}\nonumber&\\	
&\qquad\qquad=\frac{ k+2}{(k+2 M-2) (2 (\sqrt{(M-2)
   (k+M)}+M-1)+k)}.&	
\end{flalign} 
The result now follows by combining Eqs.\eqref{theoremerror2} and \eqref{theoremerror3}.
\end{proof}
Below, we plot the right side of Inequality \eqref{j=1,k=2} for $k=2$, which can be viewed as a kind of relative error.

\begin{figure}[H]
    \centering
        \begin{subfigure}[t]{0.48\textwidth}
            \centering
            \includegraphics[width=\textwidth]{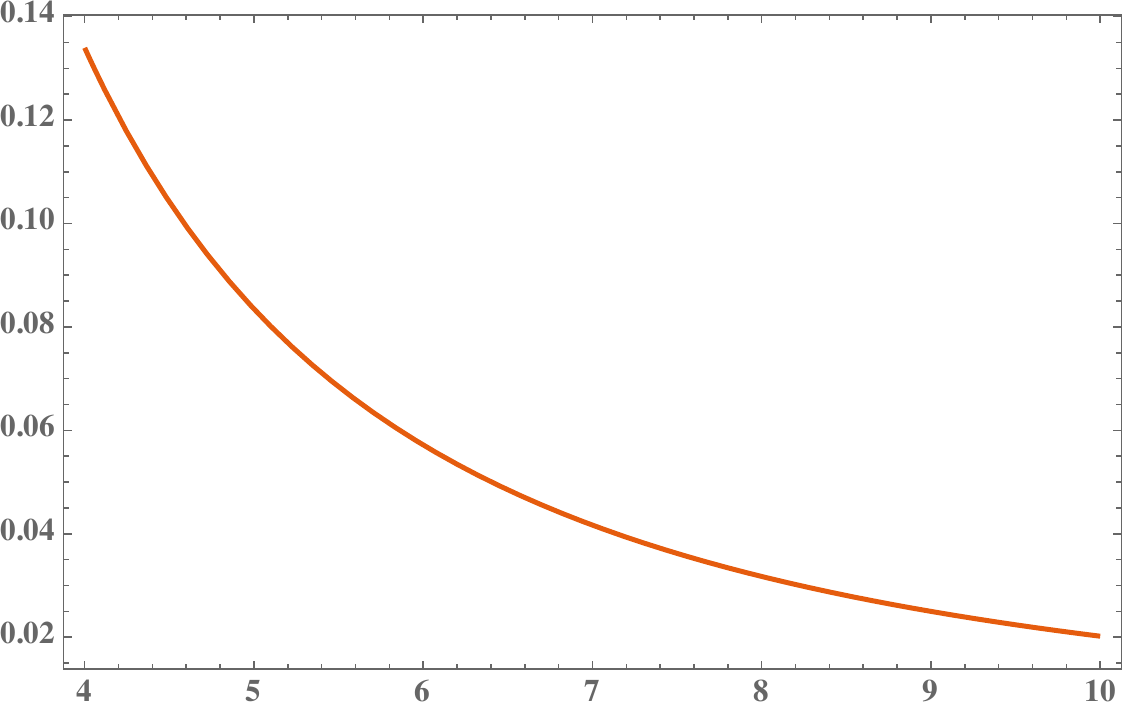}
            \caption{right side of Inequality \eqref{j=1,k=2} for $ \mu_2$ with $j=1$}
            \label{fig:mod150lattip}
        \end{subfigure}
        \hfill
        \begin{subfigure}[t]{0.48\textwidth}
            \centering
            \includegraphics[width=\textwidth]{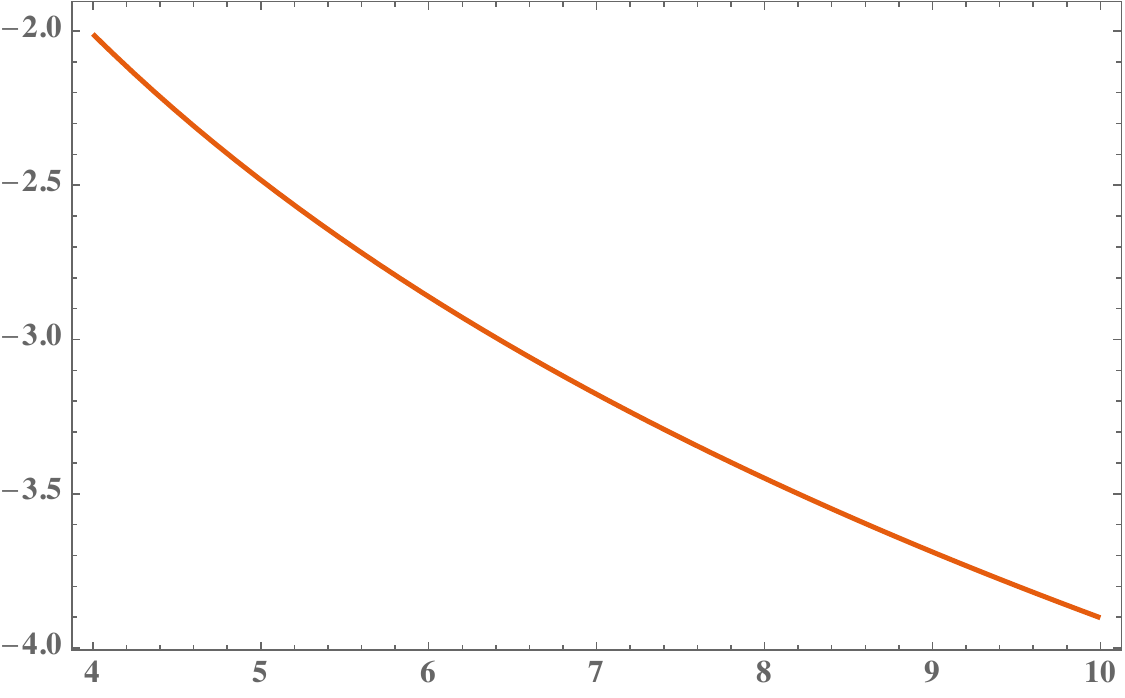}
            \caption{Log of right side of Inequality \eqref{j=1,k=2} for $ \mu_2$ with $j=1$}
            \label{fig:mod250lattip}
        \end{subfigure}
        \caption{
            Error bound according to Corollary \ref{coeffcor} for $\mu_2$ for values of $M$ on the horizontal axes. The vertical axis for the graph on the left is the right side of Inequality \eqref{j=1,k=2}, and the vertical axis on the right is the logarithm of the right side of Inequality \eqref{j=1,k=2}.
        }
        \label{fig:k2j1}
\end{figure}

\begin{corollary}\label{coeffcorj=2} Under the assumptions of Theorem \ref{mubound} with $j=2$.  Then, in the notation of  Eqs.\eqref{expansion}, \eqref{fg}, for $M\geq 4$,
\begin{flalign}\label{j=2,k=2}
&\frac{1}{A_2B_2}\left|\mu_k - \sum_{m=0}^{M}\sum_{\ell=m}^{k+m}\alpha_{k+2m-\ell}\beta_\ell A_{m,k+2m,\ell}\right|\nonumber&\\	
&\quad\leq4\frac{3 (k+2 M-4)^2 \left[(k+2 M-4) \log
   \left(\frac{k+M-1}{M-3}\right)-2 (k+2)\right]+4
   (k+2)^3}{9 (k+2)^2 (k+2 M-4)^3\sqrt{(M-3)(k+M-1)}}.&	
\end{flalign} 
\end{corollary}
\begin{proof}
From Eq.\eqref{theoremerror'} with $j=2$,
\begin{flalign}\label{theoremerror4}
&\left|\mu_k - \sum_{m=0}^{M}\sum_{\ell=m}^{k+m}\alpha_{k+2m-\ell}\beta_\ell A_{m,k+2m,\ell}\right|\nonumber&\\
&\qquad\qquad\leq \sum_{m=M+1}^{\infty}\int_{m-1}^{k+m+1}\frac{A_2B_2 dx}{(k+2m-x-2)^{\frac{5}{2}}(x-2)^{\frac{5}{2}}}& \\
&\qquad\qquad=4A_2B_2  \sum_{m=M+1}^{\infty}\frac{(k+2) \left(k^2+4 k (3 m-8)+12 (m-4) m+40\right)}{3 ((m-3)
   (k+m-1))^{3/2} (k+2 m-4)^4}\nonumber&\\
&\qquad\qquad\leq 4A_2B_2 \int_{M}^{\infty}\frac{(k+2) \left(k^2+4 k (3 m-8)+12 (m-4) m+40\right)dm}{3 ((m-3)
   (k+m-1))^{3/2} (k+2 m-4)^4},\nonumber&			
\end{flalign}
where in the last step, we require $M\geq 4$.  This last integral may be bounded as follows,
\begin{flalign}\label{theoremerror5}
&\int_{M}^{\infty}\frac{(k+2) \left(k^2+4 k (3 m-8)+12 (m-4) m+40\right)dm}{3 ((m-3)
   (k+m-1))^{3/2} (k+2 m-4)^4}&\\\nonumber	
&\leq\frac{k+2}{3\sqrt{(M-3)(k+M-1)}}\int_{M}^{\infty}\frac{\left(k^2+4 k (3 m-8)+12 (m-4) m+40\right)dm}{(m-3)
   (k+m-1) (k+2 m-4)^4}&\\\nonumber	
   &=\frac{k+2}{3\sqrt{(M-3)(k+M-1)}}\\\nonumber
  & \times\left[-\frac{2}{(k+2)^2 (k+2 M-4)}+\frac{4}{3 (k+2
   M-4)^3}+\frac{\log
   \left(\frac{k+M-1}{M-3}\right)}{(k+2)^3}\right]&\\\nonumber
    &=\frac{k+2}{3\sqrt{(M-3)(k+M-1)}}\\\nonumber
    &\times\left[\frac{3 (k+2 M-4)^2 \left((k+2 M-4) \log
   \left(\frac{k+M-1}{M-3}\right)-2 (k+2)\right)+4
   (k+2)^3}{3 (k+2)^3 (k+2 M-4)^3}\right]&\\\nonumber
   &=\frac{3 (k+2 M-4)^2 \left((k+2 M-4) \log
   \left(\frac{k+M-1}{M-3}\right)-2 (k+2)\right)+4
   (k+2)^3}{9 (k+2)^2 (k+2 M-4)^3\sqrt{(M-3)(k+M-1)}}.&		
\end{flalign} 
The result now follows by combining Eqs.\eqref{theoremerror4} and \eqref{theoremerror5}.
\end{proof}
Below we plot the right side of Inequality \eqref{j=2,k=2} for $k=2$, which can be viewed as a kind of relative error.

\begin{figure}[H]
    \centering
        \begin{subfigure}[t]{0.48\textwidth}
            \centering
            \includegraphics[width=\textwidth]{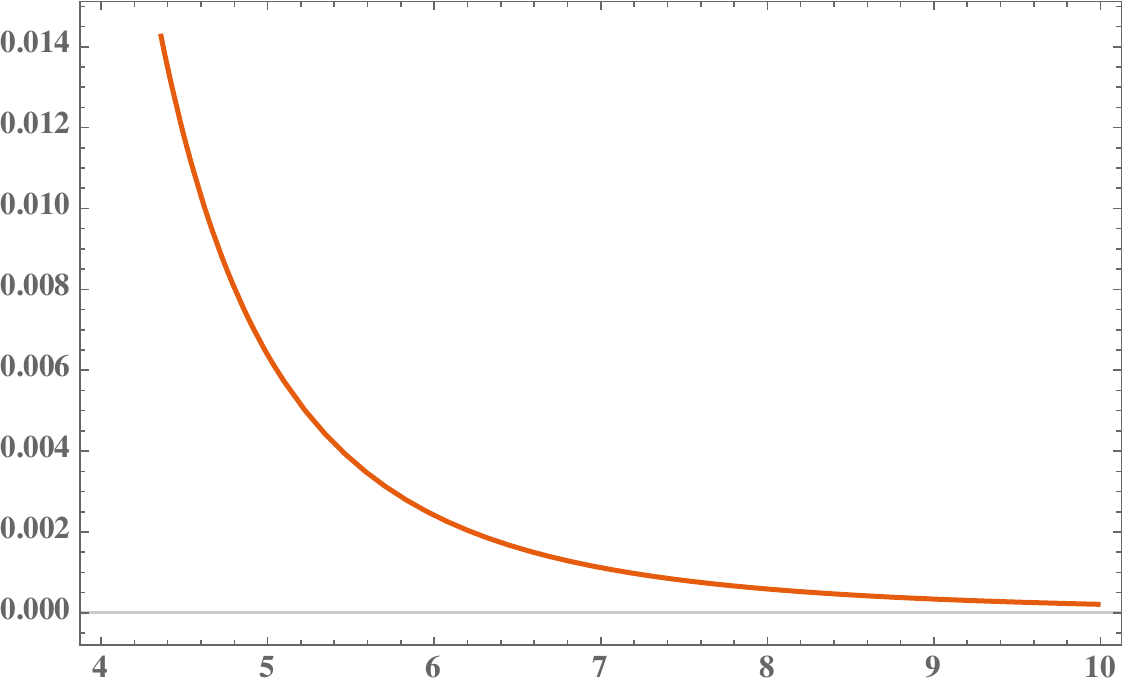}
            \caption{Right side of Inequality \eqref{j=2,k=2} for $ \mu_2$ with $j=2$}
            \label{fig:mod150lattip_2}
        \end{subfigure}
        \hfill
        \begin{subfigure}[t]{0.48\textwidth}
            \centering
            \includegraphics[width=\textwidth]{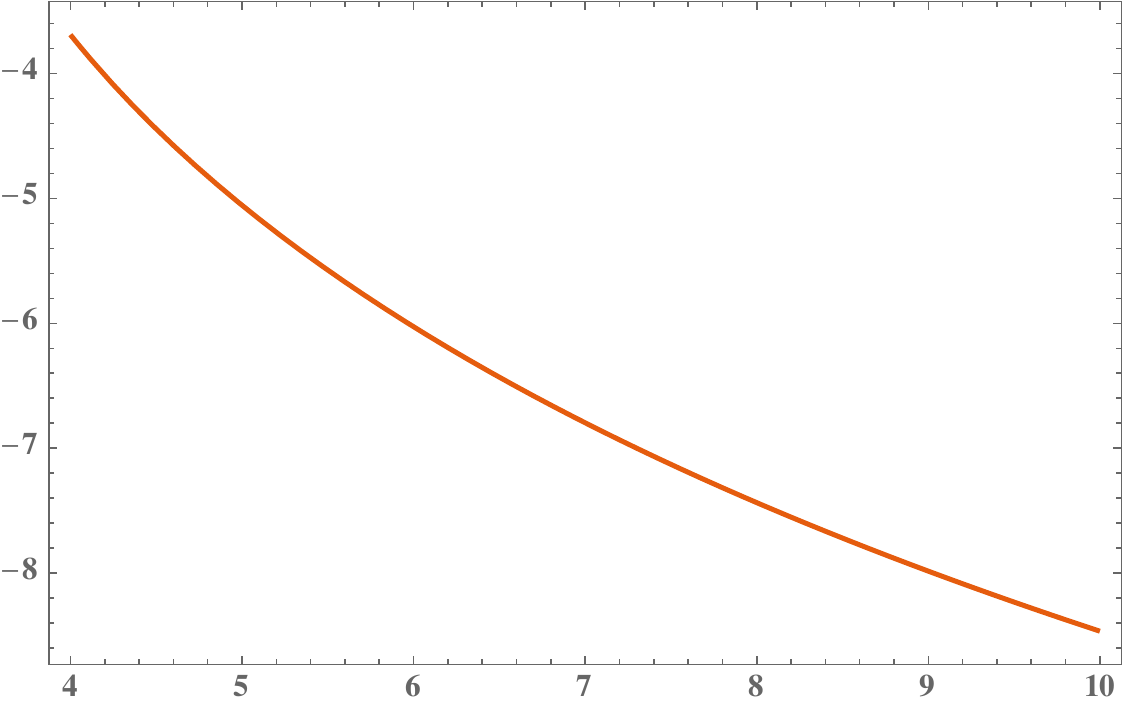}
            \caption{[Log of right side of Inequality \eqref{j=2,k=2}  for $ \mu_2$ with $j=2$}
            \label{fig:mod250lattip_2}
        \end{subfigure}
        \caption{
            Error bound according to Corollary \ref{coeffcorj=2} for $\mu_2$ for values of $M$ on the horizontal axes. The vertical axis for the graph on the left is the right side of Inequality \eqref{j=2,k=2}, and the vertical axis on the right is the logarithm of the right side of Inequality \eqref{j=2,k=2}.
        }
        \label{fig:k2j2}
\end{figure}

\section{Application}
\label{sec:application}

Our goal is to use Legendre polynomial expansions given by Eqs. \eqref{expansion}-\eqref{change-var} to solve semi-analytically a class of nonlinear partial differential equations with polynomial nonlinearity of degree 2.

\subsection{Model Prototype}
\label{sec:model}

Throughout this section, we consider the following class of \emph{nonlinear} initial boundary value problem (IBVP),
\begin{numcases}{(\text{IBVP})\quad}
    \frac{\partial T}{\partial t}
        =
            \frac{\partial}{\partial x}
                \left[
                    (1-x^2)\frac{\partial T}{\partial x}
                \right]
            +c\,T^2+f(x,t)  \label{model_eq}
    \\[1ex]
    \frac{\partial}{\partial x}T(-1,t)
        =
        \frac{\partial}{\partial x}T(1,t)
        =
        0   \label{bound_cond}
    \\[1ex]
    T(x,0)=g(x) \label{gen_init_cond}
\end{numcases}
\noindent
where $f(x,t)$, $g(x)$ are two sufficiently regular functions and $c$ is a given positive constant. The nonlinear term here is the quadratic monomial $T^2$. Note that although the IBVP considers homogeneous Neumann conditions given by \eqref{bound_cond}, the proposed solution methodology can accommodate other boundary conditions including Dirichlet and Robin-type conditions.

\subsection{Solution Methodology}
\label{sec:solution_method}

Our approach here is to use the Legendre expansions given by Eqs. \eqref{expansion} and \eqref{change-var} and reformulate IBVP as initial value problem that incurs a system of ordinary differential equations, which is a more simple problem from a numerical view point. To this end, we assume the functions $f$ and $g$ to admit the Legendre expansions:
\begin{equation}
        f(x,t)
            =
            \sum_{n=0}^{\infty} d_n(t)P_n(x),   
            \qquad
            g(x)
            =
            \sum_{n=0}^{\infty} c_nP_n(x),   
            \label{f_decomp}
\end{equation}
We also assume that the sought-after solution $T$ of IBVP can be represented by the following Legendre series:
\begin{equation}
        T(x,t)
            =
                \sum_{n=0}^{\infty} a_n(t)P_n(x).   \label{T_decomp}   
\end{equation}
Hence, using the truncated Legendre expansion for $T^2$ yields:
\begin{equation}
        \Big(
            T(x,t)
        \Big)^2
            \approx
                \sum_{n=0}^{N} b_n(t) P_n(x),   \label{T2_decomp}
\end{equation}
where the $b_n$ are the coefficients in the Fourier--Legendre expansion given by Eq.\eqref{change-var}.

Consequently, substituting expansions \eqref{f_decomp}-\eqref{T2_decomp} into IVBP allows the determination of the Legendre coefficients $a_n$ by solving the following initial value problem (IVP),
\begin{equation}\label{init_val_prob}
    (IVP) \quad
    \begin{cases}
        a_n'(t)
            &=
                -n(n+1)a_n(t)+c\, b_n(t)+d_n(t)  \\[1ex]
        a_n(0)
            &=
            c_n
    \end{cases}
    \quad n= 0,\cdots, N
\end{equation}
The resulting IVP is a system of nonlinear ODEs. Consequently, as stated earlier, the use of the Legendre expansions for both $T$ and $T^2$ results in a reduction of the numerical complexity to the requirement to solve a system of ODEs instead of a nonlinear PDE problem.  We used the Runge-Kutta method of order 4 to solve the IVP. More specifically, we used the Mathematica\textsuperscript{\tiny\textregistered} software package {\it NDSolve}. Once we numerically determine the coefficients $a_n$, we evaluate $T_{N'}(x,t)$, the partial sum of the series (\ref{T_decomp}) as follows:
\begin{equation}\label{TN_eq}
    T_{N'}(x,t)=\sum_{n=0}^{N'} a_n(t)P_n(x)
\end{equation}
where the integer $N'$ ($N'\leq N$) is chosen to be smallest integer such that the values of $T_{N'}(x,t)$ remain invariant as the values of $N'$ increase.

\subsection{Illustrative Numerical Results}
\label{sec:numeric_results}

We assess in this section the performance efficiency of the proposed solution methodology. Due to space limitations, we present results obtained in the case where the solution $T$ of IBVP is given by: 
\begin{equation}\label{solutionT}
    T(x,t) = e^{x^2-t-2}.
\end{equation}
In this case, the constant $c$ and the functions $f$ and $g$ are set to be:
\begin{equation}\label{solutionT_params}
    \begin{cases}
        \begin{aligned}
            c &= 1 \\
            f(x,t) &= e^{x^2-t-2}(4x^4+2x^2-3)-e^{2x^2-2t-4} \\
            g(x) &= e^{x^2-2}.
        \end{aligned}
    \end{cases}
\end{equation}
Other examples highlighting the salient features of the proposed solution methodology can be found in \citep{Matt}.
The obtained results in the case where $T$ is given by \eqref{solutionT} are reported in Figures \eqref{fig:Coeff_Values}-\eqref{fig:RelError_Trunc} as well as in the Appendix (see Tables \ref{tab:Table_1}-\ref{tab:Table_5}). Note that all numerical experiments reported in this section have been performed using a spatial discretization step $\Delta x = 0.005$ and a time discretization step $\Delta t = 0.01$. The following two observations are noteworthy.
\begin{itemize}
  \item Figure \eqref{fig:Coeff_Values} provides a comparison between the exact Legendre coefficients of $T$ and the computed ones obtained by solving the IVP for Legendre coefficients of low-orders $n = 0$ to $4$, medium-orders $n = 6$ to $10$, and high-orders $n = 20$ to $30$. These results show the curves corresponding to the exact values and the computed ones are undistinguishable at all times $t$. The results reported in Tables \ref{tab:Table_1}-\ref{tab:Table_5} demonstrate that the proposed solution methodology recovers the Legendre coefficient values with an impressive accuracy level. More specifically, as demonstrated by the absolute error values, the computed and the exact coefficients have identical number of digits ranging from 15 to 40 depending on the order of the coefficients and the computational time. Note that the Legendre coefficients tend to  rapidly decrease to zero with respect to both the order $n$ and the time $t$. Moreover, we have not represented the coefficients corresponding to odd values of $n$ as they are all zero since the solution $T$ given by \eqref{solutionT} is an even function with respect of the spatial variable $x$.
 
 \item Figure \eqref{fig:Trunc_Series} depicts a comparison between the exact solution $T$ and the truncated sum $T_{N'}$ given by \eqref{TN_eq} for different values of $N'$ and at different times $t$ represented as multiples of the time step $\Delta t = 0.01$. Note that at $t = 4000 \Delta t$, the solution $T$ reaches its equilibrium which is $0$ and therefore there is no need to go further in time. These results reveal that using only $6$ terms in the truncated sum ($N' = 6$) allows us to retrieve the solution $T$ with $T_{N'}$ at all times with an impressive accuracy level, as reported in Figure \eqref{fig:RelError_Trunc}. Indeed, Figure \eqref{fig:RelError_Trunc} depicts, at each time $t^{m} = m \Delta t$ (m = 100, 500, 1000, 2000, 3000, 4000), the effect of $N'$, the number of terms left in the partial sum given by \eqref{TN_eq}, on the relative error given by: 
 \begin{equation}
\displaystyle\frac{\| T(x,t^m) - T_{N'}(x,t^m) \|_2}{\| T(x,t^m) \|_2} = \displaystyle\frac{\left(\displaystyle\sum_{j = 1}^{M}|T(x_j,t^m) - T_{N'}(x_j,t^m) |^2\right)^{1/2}}{\left(\displaystyle\sum_{j = 1}^{M}|T(x_j,t^m)|^2\right)^{1/2}},\label{relerror}
\end{equation}
where $T$ (resp. $T_{N'}$) is given by \eqref{solutionT} (resp. \eqref{TN_eq}) and $x_j = j \Delta x$, with $\Delta x$ being the spatial step. We point out that such fast convergence of the Fourier--Legendre series to the exact solution with this high accuracy level is consistent with what has been  observed and reported in the literature \citep{Cohen-Tan}. Indeed, Fourier--Legendre polynomial series appear to converge much more rapidly than Taylor expansion. More specifically, it was observed in \citep{Cohen-Tan} that a sixth-order Fourier--Legendre polynomial approximation yields an error at least an order magnitude smaller than that of the analogous Taylor series polynomial.
\end{itemize}

\begin{figure}[H]
    \centering
        \begin{subfigure}[t]{0.48\textwidth}
            \centering
            \includegraphics[width=\textwidth]{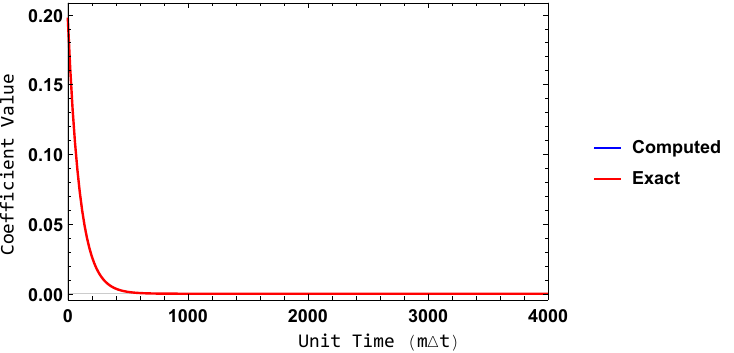}
            \caption{$n=0$}
        \end{subfigure}
        \hfill
        \begin{subfigure}[t]{0.48\textwidth}
            \centering
            \includegraphics[width=\textwidth]{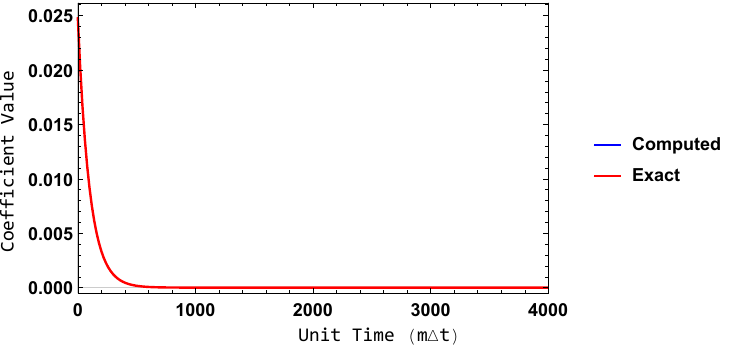}
            \caption{$n=4$}
        \end{subfigure}
        \\
        \begin{subfigure}[t]{0.48\textwidth}
            \centering
            \includegraphics[width=\textwidth]{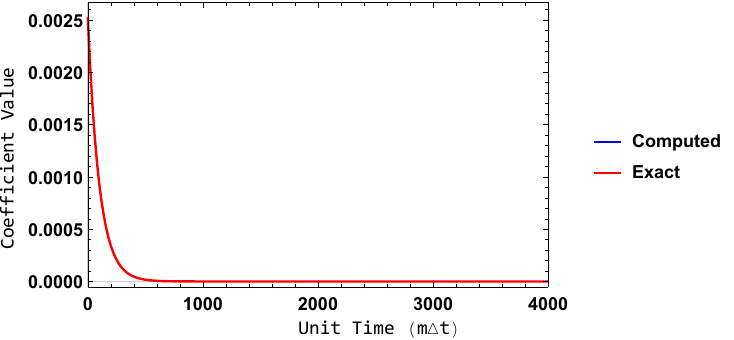}
            \caption{$n=6$}
        \end{subfigure}
        \hfill
        \begin{subfigure}[t]{0.48\textwidth}
            \centering
            \includegraphics[width=\textwidth]{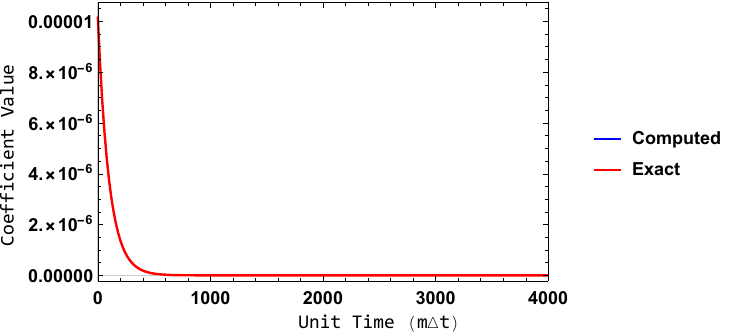}
            \caption{$n=10$}
        \end{subfigure}
        \\
        \begin{subfigure}[t]{0.48\textwidth}
            \centering
            \includegraphics[width=\textwidth]{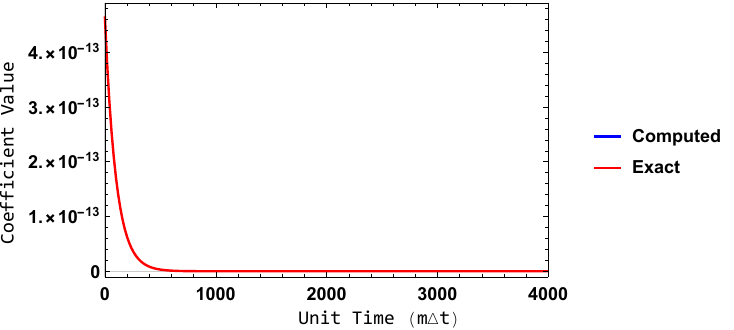}
            \caption{$n=20$}
        \end{subfigure}
        \hfill
        \begin{subfigure}[t]{0.48\textwidth}
            \centering
            \includegraphics[width=\textwidth]{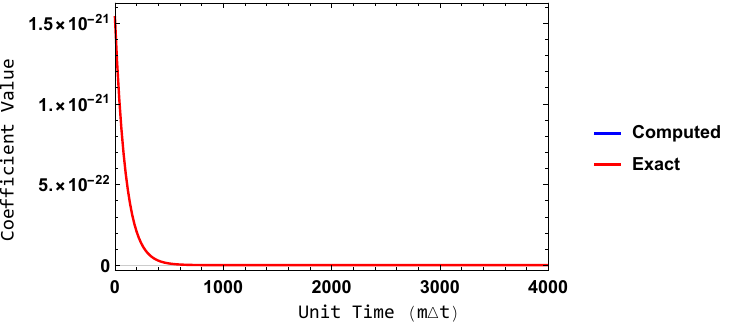}
            \caption{$n=30$}
        \end{subfigure}
    \caption{Legendre coefficient values as functions of time $a_n(t)$: exact vs computed}
    \label{fig:Coeff_Values}
\end{figure}
\begin{figure}[H]
    \centering
        \begin{subfigure}[t]{0.48\textwidth}
            \centering
            \includegraphics[width=\textwidth]{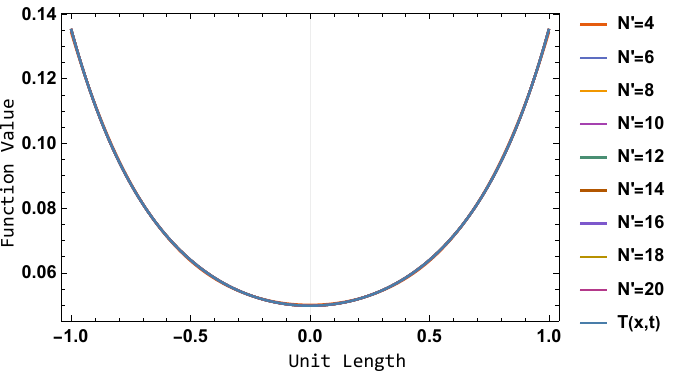}
            \caption{$t=100\Delta t$}
        \end{subfigure}
        \hfill
        \begin{subfigure}[t]{0.48\textwidth}
            \centering
            \includegraphics[width=\textwidth]{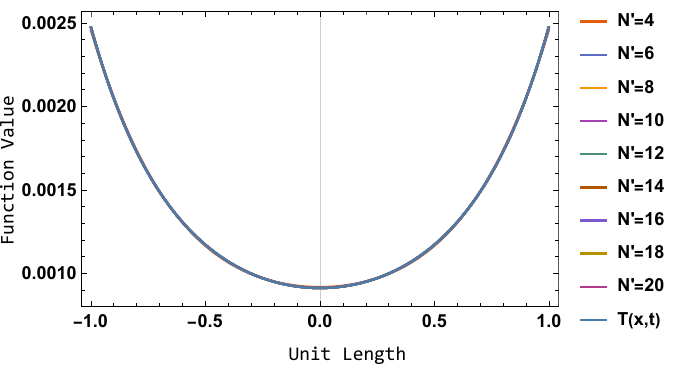}
            \caption{$t=500\Delta t$}
        \end{subfigure}
        \\
        \begin{subfigure}[t]{0.48\textwidth}
            \centering
            \includegraphics[width=\textwidth]{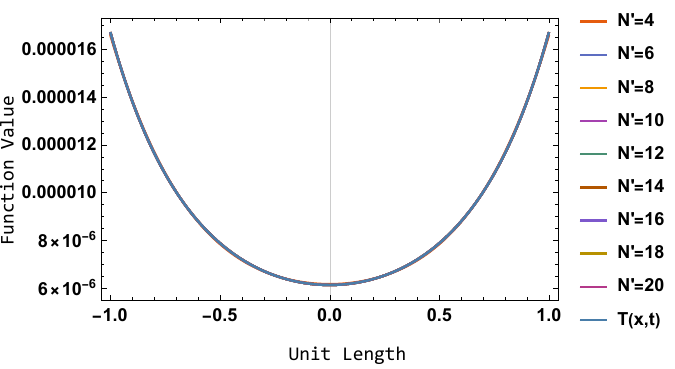}
            \caption{$t=1000\Delta t$}
        \end{subfigure}
        \hfill
        \begin{subfigure}[t]{0.48\textwidth}
            \centering
            \includegraphics[width=\textwidth]{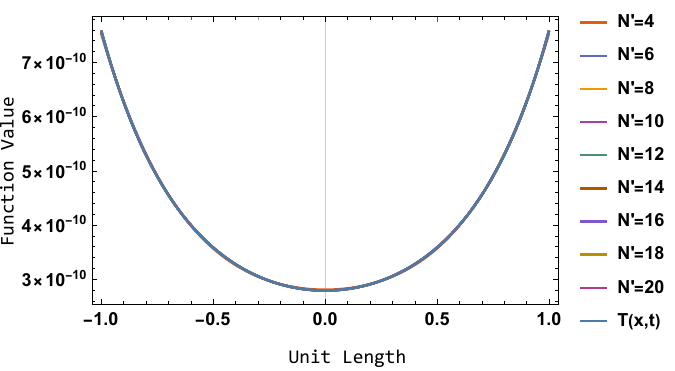}
            \caption{$t=2000\Delta t$}
        \end{subfigure}
        \\
        \begin{subfigure}[t]{0.48\textwidth}
            \centering
            \includegraphics[width=\textwidth]{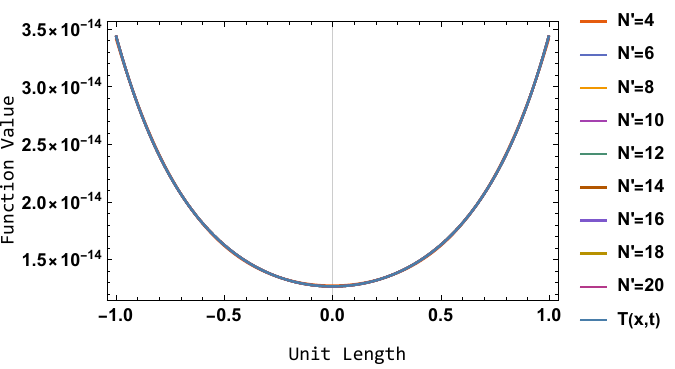}
            \caption{$t=3000\Delta t$}
        \end{subfigure}
        \hfill
        \begin{subfigure}[t]{0.48\textwidth}
            \centering
            \includegraphics[width=\textwidth]{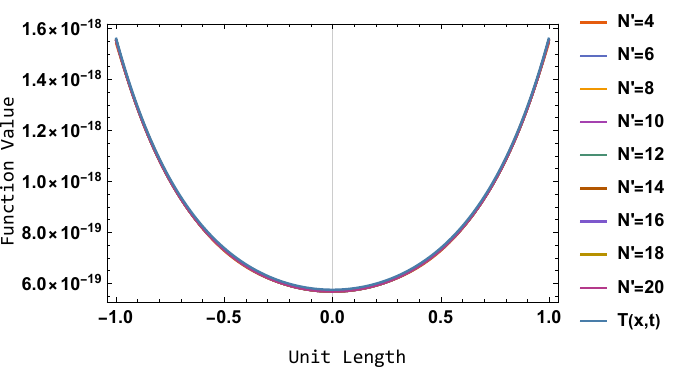}
            \caption{$t=4000\Delta t$}
        \end{subfigure}
    \caption{Exact solution of IBVP given by Eq.\eqref{solutionT} vs. Truncated computed series given by Eq.\eqref{TN_eq}. Sensitivity to the sum truncation $N'$ at various times}
    \label{fig:Trunc_Series}
\end{figure}
\begin{figure}[H]
    \centering
        \begin{subfigure}[t]{0.48\textwidth}
            \centering
            \includegraphics[width=\textwidth]{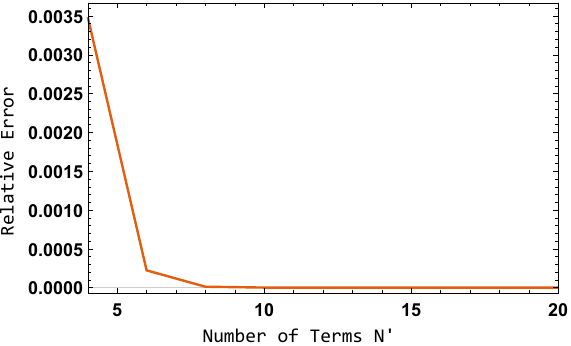}
            \caption{$t=100\Delta t$}
        \end{subfigure}
        \hfill
        \begin{subfigure}[t]{0.48\textwidth}
            \centering
            \includegraphics[width=\textwidth]{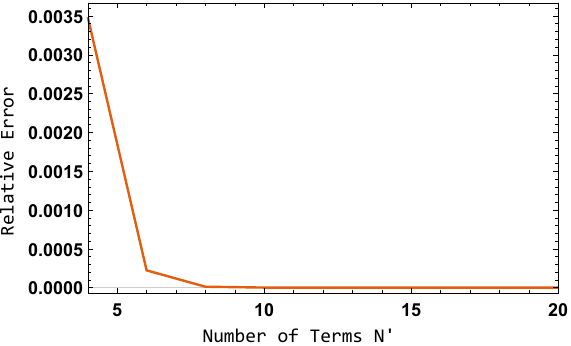}
            \caption{$t=500\Delta t$}
        \end{subfigure}
        \\
        \begin{subfigure}[t]{0.48\textwidth}
            \centering
            \includegraphics[width=\textwidth]{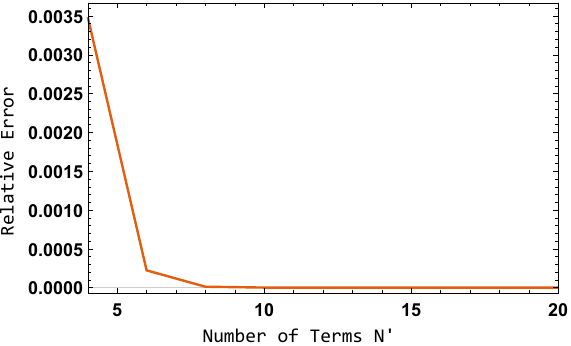}
            \caption{$t=1000\Delta t$}
        \end{subfigure}
        \hfill
        \begin{subfigure}[t]{0.48\textwidth}
            \centering
            \includegraphics[width=\textwidth]{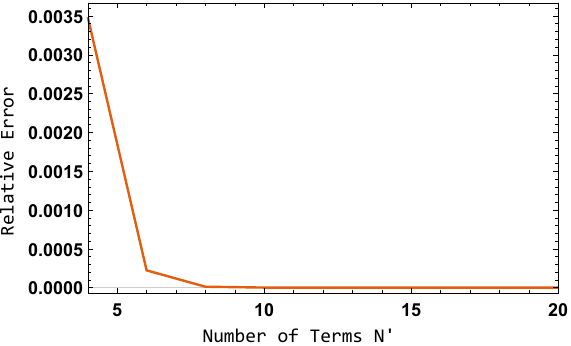}
            \caption{$t=2000\Delta t$}
        \end{subfigure}
        \\
        \begin{subfigure}[t]{0.48\textwidth}
            \centering
            \includegraphics[width=\textwidth]{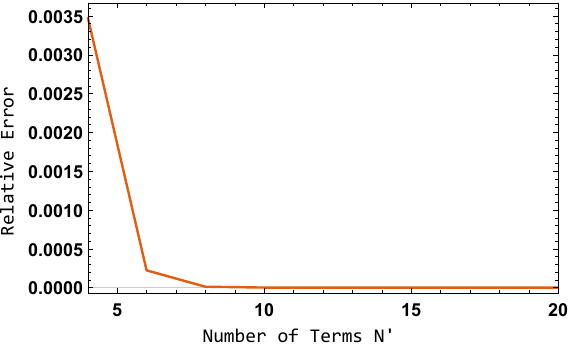}
            \caption{$t=3000\Delta t$}
        \end{subfigure}
        \hfill
        \begin{subfigure}[t]{0.48\textwidth}
            \centering
            \includegraphics[width=\textwidth]{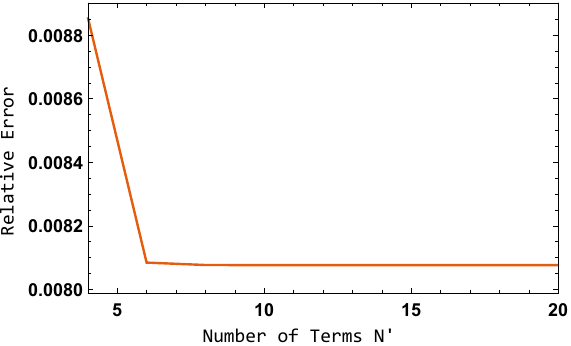}
            \caption{$t=4000\Delta t$}
        \end{subfigure}
    \caption{Sensitivity of the relative error given by Eq.\eqref{relerror} to the truncation number $N'$ at various times.}
    \label{fig:RelError_Trunc}
\end{figure}

% \clearpage

\section{Concluding Remarks}
\label{sec:discussion}

With mild restrictions on the functions $f$ and $g$, with respective Fourier--Legendre coefficients $\{\alpha_n\}$ and  $\{\beta_n\}$, we have shown  that the Fourier--Legendre coefficients $\{\mu_k\}$ of $f\cdot g$ are given by Eq.\eqref{muformula}. A bound on the rate of convergence of that series, depending on the smoothness of $f$ and $g$, is given by Theorem \ref{mubound} together with its two corollaries, and bounds on the rate of convergence of the Fourier--Legendre series of $f\cdot g$ are given in Theorems \ref{newstone} and \ref{newstone2}.  Our formulas may be iterated in a straightforward way to determine the Fourier--Legendre coefficients of $f^p(x)$ in terms of those for $f(x)$ for any positive integer $p$.

Our motivation for proving the results of Sections \ref{sec:product} and \ref{sec:conv_rates} was to apply them to finding solutions to PDEs with polynomial nonlinearities.  To that end, we demonstrated in Section \ref{sec:application} that the search for the solution of such a PDE in the form of a Fourier--Legendre series leads to an ordinary differential equation (ODE) system in which the unknowns are the Fourier--Legendre coefficients, just as for the case of linear PDEs. In contrast to the linear case, the resulting ODE in the nonlinear PDE case is also nonlinear.  Difficulties in solving this nonlinear ODE are overcome in large measure by the rapid convergence of the solution series and of the corresponding series for the nonlinear term. We showed that a high degree of accuracy is achieved by solving this system for only a few coefficients, essentially a reduced, or truncated, ODE system. We note that the resulting {\it reduced} ODE system can be solved by any preferred numerical scheme for ODEs. To our knowledge, this is the first use of Fourier-type expansions to solve semi-analytically a nonlinear PDE problem.

This example, along with other equations with non constant coefficients studied by the third listed author of this paper in \citep{Matt}, 
%highlights the potential of this semi-analytical solution methodology for providing solutions that can
include reference solutions to assess the accuracy of our semi-analytical solution methodology. 
 %when solving this class of equations with non constant coefficients. 
 We anticipate that this approach will be successful in solving PDEs with polynomial nonlinearities such as PDEs arising in global climate models.  In such models, outgoing long wave radiation is often modeled by a linear function of temperature  (the unknown function of position and time), see e.g. \citep{Cord}, but it is more accurately modeled by a multiple of the fourth power of temperature, according to the Stefan-Boltzmann law, and lateral heat diffusion may be effectively modeled by the same term as in the nonlinear PDE we solved in Section \ref{sec:application} \citep{stocker, pier}.\\

\noindent \textbf{Acknowledgement.}  The authors thank Cord Perillo for helpful discussions related to this research. 

%% The Appendices part is started with the command \appendix;
%% appendix sections are then done as normal sections
\appendix

\section{Exact and Computed Legendre Coefficients Numerical Values}
\label{sec:appendix}

\begin{table}[H]
    \centering
    \caption{Legendre coefficients $a_n(t)$: computed values vs. exact values at $t = 100 \Delta t$.}
    \begin{tabular}{ |c|c|c|c| }
        \hline
        $n$ & Computed Coefficient & Exact Coefficient & Absolute Error \\
        \hline
        0 & \num[round-mode=figures,round-precision=11]{0.072821142471856383992472095690} & \num[round-mode=figures,round-precision=11]{0.0728211424718564109009116524852} & \num[round-mode=figures,round-precision=11]{2.69084395567951031373309419004531717011958e-17} \\
        \hline
        4 & \num[round-mode=figures,round-precision=11]{0.0091378996988565902418344759515} & \num[round-mode=figures,round-precision=11]{0.00913789969885659361820919398437} & \num[round-mode=figures,round-precision=11]{3.376374718032892221722748889831885259765e-18} \\
        \hline
        6 & \num[round-mode=figures,round-precision=11]{0.00093008637272723495154814187501} & \num[round-mode=figures,round-precision=11]{0.000930086372727235295320995114579} & \num[round-mode=figures,round-precision=11]{3.437728532395676606884468649603114327283e-19} \\
        \hline
        10 & \num[round-mode=figures,round-precision=11]{3.7470012092600446191610703530964062247682298686329663e-6} & \num[round-mode=figures,round-precision=11]{3.7470012092600460940112916191864256494993644518576797e-6} & \num[round-mode=figures,round-precision=11]{1.4748502212660900194247311345832247134e-21} \\
        \hline
        20 & \num[round-mode=figures,round-precision=11]{1.71028068936320725252856781631498073739272066e-13} & \num[round-mode=figures,round-precision=11]{1.710280689689486260347003665953595950588504787e-13} & \num[round-mode=figures,round-precision=11]{3.26279007818435849638615213195784127e-23} \\
        \hline
        30 & \num[round-mode=figures,round-precision=11]{5.4017117312571501335483814440855171720615961785598720055e-22} & \num[round-mode=figures,round-precision=11]{5.6758754500932144154069619602992031525586485863295864032e-22} & \num[round-mode=figures,round-precision=11]{2.741637188360642818585805162136859804970524077697143977e-23} \\
        \hline
    \end{tabular}
    \label{tab:Table_1}
\end{table}
\begin{table}[H]
    \centering
    \caption{Legendre coefficients $a_n(t)$: computed values vs. exact values at $t = 500 \Delta t.$}
    \begin{tabular}{ |c|c|c|c| }
        \hline
        $n$ & Computed Coefficient & Exact Coefficient& Absolute Error \\
        \hline
        0 & \num[round-mode=figures,round-precision=11]{0.00133376574897958537283077599380} & \num[round-mode=figures,round-precision=11]{0.00133376574897958557281771399036} & \num[round-mode=figures,round-precision=11]{1.999869379965568818230164638438131314806e-19} \\
        \hline
        4 & \num[round-mode=figures,round-precision=11]{0.000167366471085730156896564607845} & \num[round-mode=figures,round-precision=11]{0.000167366471085730181093757290303} & \num[round-mode=figures,round-precision=11]{2.41971926824583681647896718932741620347e-20} \\
        \hline
        6 & \num[round-mode=figures,round-precision=11]{0.0000170351261382046620151160470347} & \num[round-mode=figures,round-precision=11]{0.0000170351261382046644780193931832} & \num[round-mode=figures,round-precision=11]{2.4629033461484955684339814889145822979e-21} \\
        \hline
        10 & \num[round-mode=figures,round-precision=11]{6.8628721064457290440265333676236926748468677964719e-8} & \num[round-mode=figures,round-precision=11]{6.8628721064457300392522038225083472718193702394965e-8} & \num[round-mode=figures,round-precision=11]{9.952256704548846545969725024430246e-24} \\
        \hline
        20 & \num[round-mode=figures,round-precision=11]{3.132488351061848141072794662776849295349155e-15} & \num[round-mode=figures,round-precision=11]{3.1324883510727869573849759738470035825120988e-15} & \num[round-mode=figures,round-precision=11]{1.09388163121813110701542871629438e-26} \\
        \hline
        30 & \num[round-mode=figures,round-precision=11]{1.038653316178360932735354780883601307085640600088895275e-23} & \num[round-mode=figures,round-precision=11]{1.039572851213388970768062811802427429539980508067860425e-23} & \num[round-mode=figures,round-precision=11]{9.1953503502803803270803091882612245433990797896515e-27} \\
        \hline
    \end{tabular}
    \label{tab:Table_2}
\end{table}
\begin{table}[H]
    \centering
    \caption{Legendre coefficients $a_n(t)$: computed values vs. exact values at $t = 1000 \Delta t$.}
    \begin{tabular}{ |c|c|c|c| }
        \hline
        $n$ & Computed Coefficient & Exact Coefficient & Absolute Error \\
        \hline
        0 & \num[round-mode=figures,round-precision=11]{8.9868429258199637157181371599404437173618146187767105e-6} & \num[round-mode=figures,round-precision=11]{8.9868429258199790087887016990431888051955186219286208e-6} & \num[round-mode=figures,round-precision=11]{1.52930705645391027450878337040031519103e-20} \\
        \hline
        4 & \num[round-mode=figures,round-precision=11]{1.1277064115996192534411090139132910051749139091761537e-6} & \num[round-mode=figures,round-precision=11]{1.1277064115996202717823269774733798449822192044313034e-6} & \num[round-mode=figures,round-precision=11]{1.0183412179635600888398073052952551497e-21} \\
        \hline
        6 & \num[round-mode=figures,round-precision=11]{1.147817770419584173891613758214580526668914949954205e-7} & \num[round-mode=figures,round-precision=11]{1.147817770419585210393338842075040978961401424083087e-7} & \num[round-mode=figures,round-precision=11]{1.036501725083860460452292486474128882e-22} \\
        \hline
        10 & \num[round-mode=figures,round-precision=11]{4.624166851473332327061536545195091459937001609877e-10} & \num[round-mode=figures,round-precision=11]{4.624166851473336502786597735963541038881512001313e-10} & \num[round-mode=figures,round-precision=11]{4.175725061190768449578944510391436e-25} \\
        \hline
        20 & \num[round-mode=figures,round-precision=11]{2.11065404847805523400835970027991620079103e-17} & \num[round-mode=figures,round-precision=11]{2.11065404847810679946908736127745436574419e-17} & \num[round-mode=figures,round-precision=11]{5.156546072766099753816495316e-31} \\
        \hline
        30 & \num[round-mode=figures,round-precision=11]{7.00454502315831265820747744889980051430083032491576e-26} & \num[round-mode=figures,round-precision=11]{7.00458677316397699754123465223983040876293013762895e-26} & \num[round-mode=figures,round-precision=11]{4.175000566433933375720334002989446209981271319e-31} \\
        \hline
    \end{tabular}
    \label{tab:Table_3}
\end{table}
\begin{table}[H]
    \centering
    \caption{Legendre coefficients $a_n(t)$: computed values vs. exact values at $t = 2000 \Delta t$.}
    \begin{tabular}{ |c|c|c|c| }
        \hline
        $n$ & Computed Coefficient & Exact Coefficient & Absolute Error \\
        \hline
        0 & \num[round-mode=figures,round-precision=11]{4.080020376115329765842001471995127339578344553117e-10} & \num[round-mode=figures,round-precision=11]{4.080020376187109078480250032226565059339095705408e-10} & \num[round-mode=figures,round-precision=11]{7.1779312638248560231437719760751152292e-21} \\
        \hline
        4 & \num[round-mode=figures,round-precision=11]{5.11977918793265898417101799879545833229252254524e-11} & \num[round-mode=figures,round-precision=11]{5.11977918793265925901494689145938461842823751679e-11} & \num[round-mode=figures,round-precision=11]{2.7484392889266392628613571497155e-27} \\
        \hline
        6 & \num[round-mode=figures,round-precision=11]{5.2110846157181128232278648583132165958096698227e-12} & \num[round-mode=figures,round-precision=11]{5.211084615718113102959874137807547166582542233e-12} & \num[round-mode=figures,round-precision=11]{2.797320092794943305707728724102e-28} \\
        \hline
        10 & \num[round-mode=figures,round-precision=11]{2.09936850266900186571923592966638457081284217e-14} & \num[round-mode=figures,round-precision=11]{2.09936850266900197841070485349869399324511533e-14} & \num[round-mode=figures,round-precision=11]{1.1269146892383230942243227316e-30} \\
        \hline
        20 & \num[round-mode=figures,round-precision=11]{9.5823545553810340301160223774413513440761440149684696207e-22} & \num[round-mode=figures,round-precision=11]{9.5823545553810345547138244300468813210133843416338289413e-22} & \num[round-mode=figures,round-precision=11]{5.245978020526055299769372403266653593206e-38} \\
        \hline
        30 & \num[round-mode=figures,round-precision=11]{3.18007747430828645434346585078122781145271419001e-30} & \num[round-mode=figures,round-precision=11]{3.18007747516874966850108273035964835398779960823e-30} & \num[round-mode=figures,round-precision=11]{8.6046321415761687957842054253508541822e-40} \\
        \hline
    \end{tabular}
    \label{tab:Table_4}
\end{table}
\begin{table}[H]
    \centering
    \caption{Legendre coefficients $a_n(t)$: computed values vs. exact values at $t = 4000 \Delta t$.}
    \begin{tabular}{ |c|c|c|c| }
        \hline
        $n$ & Computed Coefficient & Exact Coefficient & Absolute Error \\
        \hline
        0 & \num[round-mode=figures,round-precision=11]{8.337769687529857208135134433242288769193e-19} & \num[round-mode=figures,round-precision=11]{8.409548778001187301056047145477344066244e-19} & \num[round-mode=figures,round-precision=11]{7.1779090471330092920912712235055297051e-21} \\
        \hline
        4 & \num[round-mode=figures,round-precision=11]{1.055265141924529476433043608797143689468e-19} & \num[round-mode=figures,round-precision=11]{1.055265141929293856413546705953971003236e-19} & \num[round-mode=figures,round-precision=11]{4.764379980503097156827313769e-31} \\
        \hline
        6 & \num[round-mode=figures,round-precision=11]{1.07408459324720682697383231128328022066e-20} & \num[round-mode=figures,round-precision=11]{1.074084593252122890358746628151136708774304151512694457969e-20} & \num[round-mode=figures,round-precision=11]{4.9160633849143168678564881e-32} \\
        \hline
        10 & \num[round-mode=figures,round-precision=11]{4.327120994089670339886267906403958918742486093711993084e-23} & \num[round-mode=figures,round-precision=11]{4.327120994109624585557016129827564565641410461538962104e-23} & \num[round-mode=figures,round-precision=11]{1.995424567074822342360564689892436782696902e-34} \\
        \hline
        20 & \num[round-mode=figures,round-precision=11]{1.97507048032228528200571893433898454379912313313e-30} & \num[round-mode=figures,round-precision=11]{1.97507048033142355631629972708030419188863871787e-30} & \num[round-mode=figures,round-precision=11]{9.13827431058079274131964808951558475e-42} \\
        \hline
        30 & \num[round-mode=figures,round-precision=11]{6.55462820754898477267247857915629338009e-39} & \num[round-mode=figures,round-precision=11]{6.55462820757933131071572453881670915289e-39} & \num[round-mode=figures,round-precision=11]{3.034653804324595966041577281e-50} \\
        \hline
    \end{tabular}
    \label{tab:Table_5}
\end{table}

\section{Conditions in Theorem \ref{L2} are Sufficient Conditions: An illustrative Example}
\label{sec:appendixB}

In this appendix, we consider the function,
\begin{equation}\label{fdef}
f(x)= \sum_{n=1}^{\infty}\frac{P_n(x)}{n}.
\end{equation}
We analyze some of its properties and calculate Fourier--Legendre coefficients $\mu_k$ for $f^2$  (see Eq.\eqref{muformula}) to illustrate some of our results, and to demonstrate that some of the hypotheses, though sufficient, are not necessary for our results to hold. 
\begin{figure}[H]
\center
\includegraphics[width=.7\textwidth]{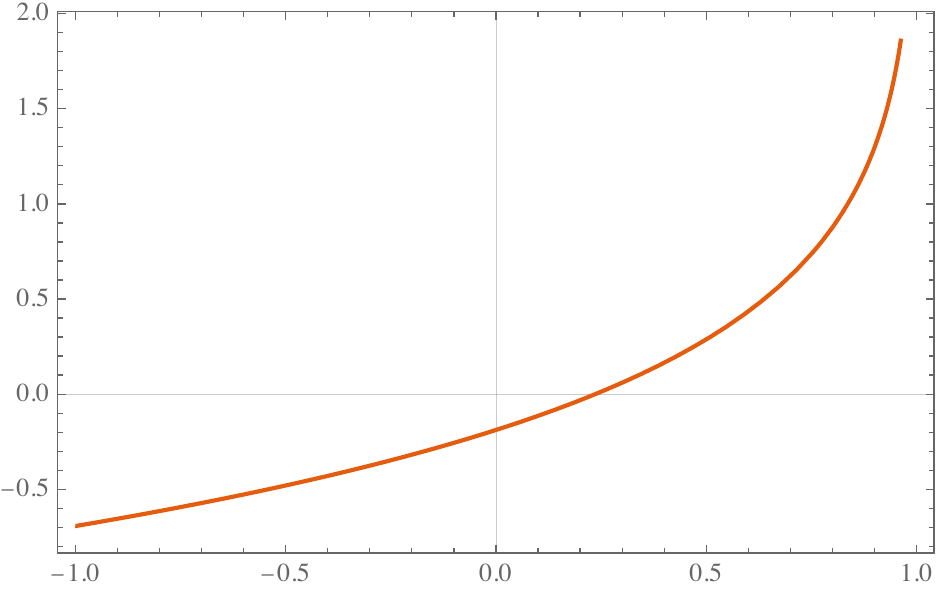}
\caption{Graph of $f(x)= \sum_{n=1}^{1000}\frac{P_n(x)}{n}$}
\label{f1000}
\end{figure}
The series in Eq.\eqref{fdef} is easily shown to converge in $L^2[-1,1]$ by expanding $f$ in the orthonormal basis $\{ \sqrt{n+\frac{1}{2}}\,P_n(x)\}$ and using Parseval's theorem.  
%It follows that $f(x)$ is defined for almost all $x\in [-1,1]$, and that there is a subsequence of the partial sums of the series in Eq.\eqref{fdef} that converges pointwise almost everywhere.  
We note that since $P_n(-1)=(-1)^n$, the series defining $f$ converges conditionally at $x=-1$, and since $P_n(1)=1$ for all $n$, it diverges at $x=1$. 

For later purposes, we let,
\begin{equation}\label{fdef}
g(x)= \sum_{n=1}^{\infty}\frac{|P_n(x)|}{n},
\end{equation}
and observe that the sum for $g$ diverges at $x=\pm1$.  According to Bernstein's inequality (see, e.g. \citep{Saxena}),
\begin{equation}
|P_n(x)|\leq\sqrt{\frac{2}{\pi n}}\,\,\frac{1}{(1-x^2)^{\frac{1}{4}}}
\end{equation}
for $x\in[-1,1]$ and $n\geq 1$. It follows that the series in Eq\eqref{fdef} converges absolutely in $(-1,1)$ and,
\begin{align}
|f(x)|\leq g(x) \leq \frac{K}{(1-x^2)^{\frac{1}{4}}}\\
f^2(x) \leq g^2(x)\leq \frac{K^2}{(1-x^2)^{\frac{1}{2}}},\label{integrable}
\end{align}
for some constant $K$. Thus, $f(x)$, $g(x)$, $f^2(x)$, and $g^2(x)$ are integrable on $[-1,1]$.
\begin{figure}[H]
\center
\includegraphics[width=.7\textwidth]{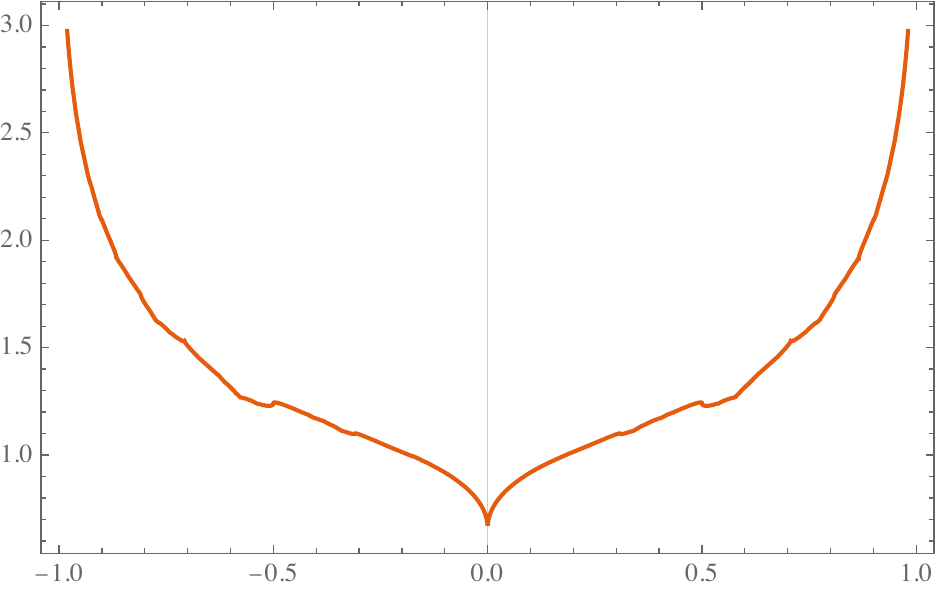}
\caption{Graph of $g(x)= \sum_{n=1}^{1000}\frac{|P_n(x)|}{n}$}
\label{f1000}
\end{figure}
Observe next that $f$ is unbounded, and therefore cannot satisfy the hypotheses of Theorem \ref{L2}.  Indeed, either $f$ is not absolutely continuous on $[-1,1)$ in which case the hypotheses are clearly not satisfied, or $f$ is  absolutely continuous on $[-1,1)$, but its derivative must be an unbounded function by the Fundamental Theorem of Calculus. Therefore, $f'$ cannot have bounded variation.  We consider next whether $f^2$ nevertheless satisfies the conclusion of Theorem \ref{L2}.

To calculate $\mu_k$ for the function $f^2$ from the formula in Eq.\eqref{muformula}, we let $\alpha_k = \beta_k=1/k$ and use Eq.\eqref{ajkl} to find that,
\begin{align}\label{formula2}
\sum_{\ell=m}^{k+m}\alpha_{k+2m-\ell}\beta_\ell A_{m,k+2m,\ell}=\sum_{\ell=m}^{k+m}\frac{A_{m,k+2m,\ell}}{\ell(k+2m-\ell)}= \frac{4 k+2}{m (k+2 m) (2 k+2 m+1)},
\end{align} for any nonnegative integer $k$.  Using Mathematica, we then calculate,
\begin{flalign}\label{muformula3}
\quad\mu_k &=\sum_{m=0}^{\infty} \sum_{\ell=m}^{k+m}\alpha_{k+2m-\ell}\beta_\ell A_{m,k+2m,\ell}\nonumber&\\
&=\sum_{m=1}^{\infty}\frac{4 k+2}{m (k+2 m) (2 k+2 m+1)}+ \sum _{\ell=1}^{k-1} \frac{A_{0,k,\ell}}{\ell (k-\ell)}&\\
&=\frac{(4 k+2) H_{\frac{k}{2}}-2 k H_{k+\frac{1}{2}}}{k (k+1)}\,+\, \sum _{\ell=1}^{k-1} \frac{A_{0,k,\ell}}{\ell (k-\ell)},&\nonumber
\end{flalign}
where we take the finite sum on $\ell$ to be zero for $k\leq1$. In the last line of Eq.\eqref{muformula3}, the $n$th Harmonic number, $H_n$, is given by,
\begin{equation}
H_n = \sum_{k=1}^n \frac{1}{k} =\int_{0}^{1}\frac{1-x^n}{1-x}dx,
\end{equation}
when $n$ is a positive integer, and $H_n$ is defined by the integral expression when $n$ is real.  Eq.\eqref{muformula3} allows us to make exact calculations, for example,
\begin{align}
\mu_0 &=\frac{1}{6} \left(-24+\pi ^2+6 \log (16)\right)\\
\mu_1&=6 \left(\frac{5}{9}-\frac{\log (4)}{3}\right)\\
\mu_2 &= \frac{13}{45}+\frac{4 \log (2)}{3}\\
\mu_3 &=\frac{641}{315}-\frac{4 \log (2)}{3}\\
\mu_{10} &= \frac{139252469}{426786360}+\frac{4 \log (2)}{11}.
\end{align}
These values may compared with direct calculations via,
\begin{align}\label{mudirect}
\mu_k=&\frac{2k+1}{2}\int_{-1}^{1}f^2(x)P_k(x)dx=\frac{2k+1}{2}\int_{-1}^{1}\left(\sum_{n=1}^{\infty}\frac{P_n(x)}{n}\right)^2 P_k(x)dx\nonumber\\ 
&\approx \frac{2k+1}{2}\int_{-1}^{1}\left(\sum_{n=1}^{N}\frac{P_n(x)}{n}\right)^2 P_k(x)dx
\end{align}
for large $N$, where the approximation is justified by the Lebesgue Dominated Convergence Theorem with $g^2$ as the integrable dominating function (see Eq.\eqref{integrable}.  

The following table compares approximations for the first few values of $\mu_k$ via the integral approximation of Eq.\eqref{mudirect}, with $N=500$ (middle column), to decimal approximations to the exact values of $\mu_k$ calculated from Eq.\eqref{muformula3} (right column). The estimates are close for higher values of $k$ as well.
\begin{table}[htp]
\caption{Comparison of integral estimates with exact (roundoff) values of $\mu_k$}  
\begin{center}
\resizebox{\textwidth}{!}{
\begin{tabular}{||c|c|c||}
\hline
k&$\mu_k\approx\frac{2k+1}{2}\int_{-1}^1\left(\sum_{n=1}^{500}\frac{P_n(x)}{n}\right)^2 P_k(x) dx$& $\mu_k =\sum_{m=0}^{\infty} \sum_{\ell=m}^{k+m}\alpha_{k+2m-\ell}\beta_\ell A_{m,k+2m,\ell}$\\
\hline
0&0.417522&0.417523\\
\hline
1&0.560742&0.560745\\
\hline
2&1.213080&1.213085\\
\hline
3&1.110717&1.110724\\
\hline
\end{tabular}
}
\end{center}\label{tableclimsen2}
\end{table}

The figure below displays the graphs of  $f^2(x)\approx\left(\sum_{n=1}^{200}\frac{P_n(x)}{n}\right)^2$ with $\sum_{k=0}^{40}\mu_k P_k(x)$, where $\mu_0, \mu_1, \dots \mu_{40}$ were computed using Eq\eqref{muformula3}.
\begin{figure}[H]
\center
\includegraphics[width=.7\textwidth]{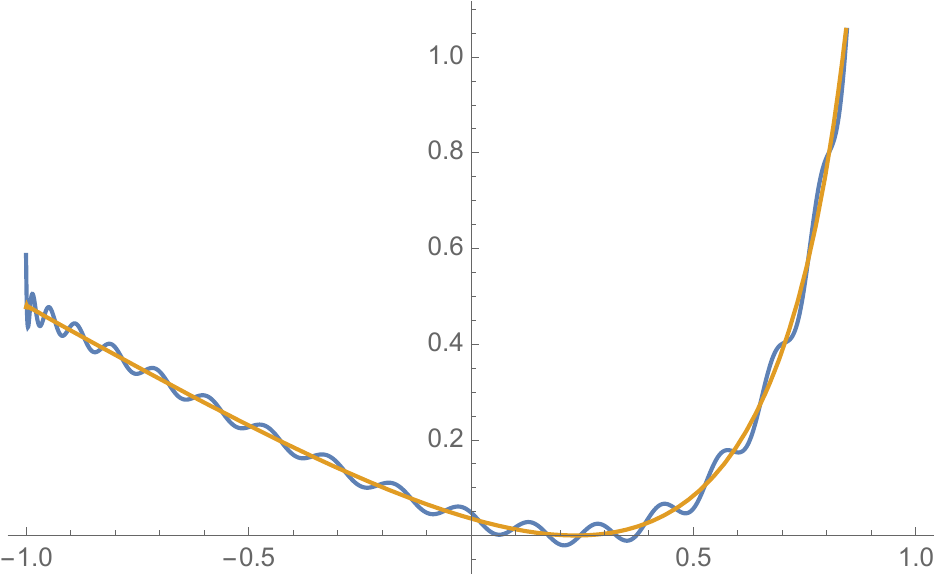}
\caption{Graph of $\left(\sum_{n=1}^{200}\frac{P_n(x)}{n}\right)^2$(orange); and $\sum_{k=0}^{40}\mu_k P_k(x)$ (blue) with $\mu_k$ computed using Eq\eqref{muformula3}.} 
\label{f1000}
\end{figure}
We see then that conclusion of Theorem \ref{L2} applies to our function $f$ even though its hypotheses are not met. We note, however, that the hypotheses to all of our theorems are easily satisfied by solutions to the class of partial differential equations with polynomial nonlinearities considered here, and in other contexts of practical interest.

\bibliographystyle{elsarticle-num-names} 
%\bibliography{cas-refs}

%% else use the following coding to input the bibitems directly in the
%% TeX file.

% \begin{thebibliography}{00}

 %% \bibitem[Author(year)]{label}
 %% Text of bibliographic item

% \bibitem[ ()]{}
%
% \end{thebibliography}

\end{document}